\documentclass[times,doublespace]{amsart}

\usepackage[pdfauthor={Congling Qiu}, 
        pdftitle={???}%
      dvips,colorlinks=true]{hyperref}
\hypersetup{
  bookmarksnumbered=true,
  citecolor=black,
  pagecolor=black, 
  urlcolor=black,  
}

\usepackage[utf8]{inputenc}

\usepackage{xtab}
\usepackage{amsmath, amssymb, cancel,verbatim}
\usepackage[all]{xy}
\usepackage[scr=euler,scrscaled=.96]{mathalfa}
 
\usepackage{enumerate}
\usepackage{mathtools,booktabs}
\usepackage{amscd}
\usepackage{bbm, stmaryrd}
\usepackage{yfonts}
\usepackage{color}

\usepackage{epstopdf} 
\usepackage{booktabs}

\newtheorem{teo}{Theorem}[subsection]
\newtheorem{thm}[teo]{Theorem}

\newtheorem{eg}[teo]{Example}
\newtheorem{lem}[teo]{Lemma}
\newtheorem{cor}[teo]{Corollary}
\newtheorem{defn}[teo]{Definition}

\newtheorem{rmk}[teo]{Remark}

\newtheorem{asmp}[teo]{Assumption}

  \newtheorem{prop}[teo]{Proposition}

\numberwithin{equation}{section}

  \newcommand{\BA}{{\mathbb {A}}} 
     
     \newcommand{\BF}{{\mathbb {F}}}
    \newcommand{\BG}{{\mathbb {G}}}

    \newcommand{\BM}{{\mathbb {M}}} \newcommand{\BN}{{\mathbb {N}}}
     
    \newcommand{\BQ}{{\mathbb {Q}}} \newcommand{\BR}{{\mathbb {R}}}

     \newcommand{\BZ}{{\mathbb {Z}}}

    \newcommand{\CA}{{\mathcal {A}}} \newcommand{\CB}{{\mathcal {B}}}
     \newcommand{\cD}{{\mathcal {D}}}
     
    \newcommand{\CG}{{\mathcal {G}}} 
    \newcommand{\CI}{{\mathcal {I}}} 
     
    \newcommand{\CM}{{\mathcal {M}}} 
    \newcommand{\CO}{{\mathcal {O}}}

    \newcommand{\CU}{{\mathcal {U}}} 
    \newcommand{\CW}{{\mathcal {W}}} \newcommand{\CX}{{\mathcal {X}}}
    \newcommand{\CY}{{\mathcal {Y}}} \newcommand{\CZ}{{\mathcal {Z}}}

    \newcommand{\fm}{{\mathfrak{m}}}

     \newcommand{\fA}{{\mathfrak{A}}} \newcommand{\fB}{{\mathfrak{B}}}

    \newcommand{\fI}{{\mathfrak{I}}} 
     
    \newcommand{\fM}{{\mathfrak{M}}}

     \newcommand{\fT}{{\mathfrak{T}}}
    \newcommand{\fU}{{\mathfrak{U}}} 
     \newcommand{\fX}{{\mathfrak{X}}}
     \newcommand{\fZ}{{\mathfrak{Z}}}

       \newcommand{\dbc}{{\circ\circ}}
    
    \newcommand{\pair}[1]{\langle {#1} \rangle}

    \newcommand{\incl}{\hookrightarrow}

    \newcommand{\bsl}{\backslash}

 \newcommand{\vep}{\varepsilon} \newcommand{\ep}{\epsilon}
  \newcommand{\vpl}{\varprojlim}
    \newcommand{\fpl}{{\flat+}}      \newcommand{\fcc}{{\flat\circ}}   
 \newcommand{\vil}{\varinjlim}

    \newcommand{\etale}{\'{e}tale~}

  \newcommand{\cycl}{{\mathrm{cycl}}}

    \newcommand{\Aut}{{\mathrm{Aut}}}

    \newcommand{\can}{{\mathrm{can}}}
    \newcommand{\Coker}{{\mathrm{Coker}}}
    
    \newcommand{\cyc}{{\mathrm{cyc}}}

    \newcommand{\et}{{\mathrm{et}}}
    
    \newcommand{\Fr}{{\mathrm{Fr}}}

    \newcommand{\Gal}{{\mathrm{Gal}}} \newcommand{\GL}{{\mathrm{GL}}}
        
    \newcommand{\GSp}{{\mathrm{GSp}}}
    \newcommand{\Hom}{{\mathrm{Hom}}}

    \newcommand{\pr}{{\mathrm{pr}}}    \newcommand{\perf}{{\mathrm{perf}}}

    \newcommand{\ST}{{\mathrm{ST}}}

    \renewcommand{\mod}{\ \mathrm{mod}\ }

    \newcommand{\red}{{\mathrm{red}}}

    \newcommand{\Spec}{{\mathrm{Spec}}} \newcommand{\Spf}{{\mathrm{Spf}}} \newcommand{\Spa}{{\mathrm{Spa}}}

    \newcommand{\Sym}{{\mathrm{Sym}}}

    \newcommand{\tor}{{\mathrm{tor}}}
    \newcommand{\univ}{{\mathrm{univ}}}

  \newcommand{\Zar}{{\mathrm{Zar}}}

\newcommand\supervisor[1]{\def\@supervisor{#1}}

\newcounter{elno}

\renewcommand{\cong}{\simeq}

 \author{Congling Qiu} 

\begin{document} 

\title{The Manin-Mumford conjecture and  the Tate-Voloch conjecture      for a product of   Siegel moduli spaces}
\maketitle

\tableofcontents  
  \begin{abstract}
We use perfectoid spaces associated to abelian varieties and Siegel moduli spaces to study torsion points and ordinary CM points. We reprove the Manin-Mumford conjecture, i.e. Raynaud's theorem. We also prove the Tate-Voloch conjecture for a product of Siegel moduli spaces, namely ordinary  CM points outside a closed subvariety can not be $p$-adically too close to it.
\end {abstract}
      
\section{Introduction} 

We  use the theory of  perfectoid spaces to study torsion points in abelian varieties and ordinary CM points in    Siegel moduli spaces.
 The use of perfectoid spaces   is  inspired by Xie's recent work  \cite{Xie}.
 \subsection{Tate-Voloch conjecture}
 Our main new result is  about  ordinary CM points.
 Let $p$ be a prime number, $L$ the complete maximal unramified extension of $\BQ_p$.   
Let $X $ be a product of Siegel moduli spaces  over $ L$ with arbitrary level structures. 
      \begin{thm} \label{TVcor1}    
       Let $Z$ be a closed subvariety  
of  $X_ {\bar  L}$. There exists a constant $c>0$ such that for every     ordinary CM point   $x\in X(\bar  L)$, if the   distance $d(x,Z)$ from $x$ to $Z$ satisfies 
$d(x,Z)\leq c$, then $x\in Z$.
        \end{thm}

        The distance $d(x,Z)$  is  defined as follows.   
           Let    $\|\cdot\|$ be a $p$-adic norm  on $\bar  L$.  
   Let $\fX$ be  an integral model  of $X$ over $\CO_{\bar  L}$. 
                           Let $\{U_1,...,U_n\}$ be a finite  open  cover of $ \fX $ by  affine schemes flat over  $\CO_{\bar  L}$.   
                             Define  $d(x,Z)$  to be the supremum of  $\|f(x)\|$'s 
                             where $  U_i$ contains $x$ and $ f\in \CO_\fX(U_i)$ vanishing on $Z\bigcap U_i$.
                  The definition of $d(x,Z)$ depends on the choices of the  integral model and  the cover. However, the  truth of Theorem \ref{TVcor1} does not depend on these choices, see \ref{The distance function}.                 Moreover,     we show  that Theorem \ref{TVcor1} holds for formal subschemes of $\fX$ (with maximal level at $p$), see Theorem \ref{TVcor1aff}.
     And for CM points which are canonical liftings, we prove an ``almost effective" version, see Theorem \ref{samel}.
                  
                   It is clear that     the same
        statement in Theorem \ref{TVcor1} is  true replacing $X_{\bar L}$ by a closed subvariety.  In particular, 
                   Theorem \ref{TVcor1}  is in fact equivalent the same statement  for $X$ being a single Siegel moduli space, by embedding a product of Siegel moduli spaces  into a larger   Siegel moduli space.

 \begin{rmk}  
 
 (1) 
  For  a power of the modular curve  without level structure,  Theorem \ref{TVcor1}    was proved by Habegger  \cite{Hab} by  a different method.      However,  Habegger's proof relies on a result of Pila \cite{Pil} (see also \cite[Theorem 8]{Hab}) concerning  Zariski closure of a Hecke orbit.  
As far as we know, it is not available for  Siegel moduli spaces yet.    
   Moreover, Habegger's  method seems  not applicable to formal schemes.
   
   
   (2)   Habegger   \cite{Hab} also showed that  the ordinary  condition  is necessary. 
   
 (3)  The original Tate-Voloch conjecture  \cite{TV}   states that in a semi-abelian variety,    torsion points outside a closed subvariety can not be $p$-adically too close to it.
This conjecture  was proved by Scanlon   \cite{Sca0}   \cite{Sca} when the semi-abelian variety is defined over $\bar \BQ_p$.  
Xie \cite{Xie} proved   dynamic  analogs of  Tate-Voloch conjecture for projective spaces.

\end{rmk}
  \subsubsection{Idea  of the proof of Theorem \ref{TVcor1}} 
 It is not hard to reduce  Theorem \ref{TVcor1} to the case that $X$ has maximal level at $p$,  see Lemma \ref{T'T}. We sketch   the proof of Theorem \ref{TVcor1} in this case.     Relative to  the  canonical lifting of an ordinary point $x$ in the reduction of $X$, ordinary CM points in $X$ with   reduction $x$ are 
  like $p$-primary roots of unity relative to 1 in the open unit disc   around 1 (see Proposition \ref{dJN3.2}). This is the  Serre-Tate theory. If we only consider one such disc,   Theorem \ref{TVcor1}   follows from a result of Serban \cite{Ser}. 
   In general,
     we  need to  study all  infinitely   
  many Serre-Tate  deformation spaces together. In characteristic $p$, this can be achieved by 
Chai's global Serre-Tate  theory   \cite{Cha} (see \ref{GST}). 
To prove Theorem \ref{TVcor1}, we at first prove a Tate-Voloch type result in a family  characteristic $p$   (see  \ref{pf{TVcor1}}).   
   Then  we use the ordinary perfectoid Siegel space associated to  $ X$  and the perfectoid universal covers of  
   Serre-Tate  deformation spaces
   to translate this result to the desired Theorem \ref{TVcor1}.

 

      \subsubsection{Possible generalizations} 
       For   Shimura varieties  of Hodge type, the ordinary locus in the usual sense could be empty.
In this case, we consider  the notion of  $\mu$-ordinariness (see \cite{Wed}).   
        Then following our strategy, we need three ingredients. 
      At first, a theory of
Serre-Tate coordinates for $\mu$-ordinary CM points. 
For Shimura varieties of Hodge type,   see \cite{Hon} and  \cite{SZ}.
 Secondly, a  
global theory of Serre-Tate coordinates   in characteristic $p$.    For Shimura varieties of PEL type, such results  should be known to experts.
 Thirdly,  $\mu$-ordinary  perfectoid   Shimura varieties.    Following  \cite{Sch13}, certain perfectoid   Shimura varieties  of abelian type are constructed in \cite{She}.
   For 
universal abelian varieties over  Shimura varieties of PEL type,  
we expect  a Tate-Voloch type result for torsion points in   fibers over $\mu$-ordinary CM points.  
Still, we need analogs of the above three  ingredients.

 \subsection{Manin-Mumford conjecture} 
For torsion points in   abelian varieties, we  reprove Raynaud's theorem   \cite{Ray83}, which is also known as   the Manin-Mumford conjecture.
 \begin{thm} [Raynaud   \cite{Ray83}] \label{MM}  Let $F$ be a number field.
 Let $A$ be an  abelian variety over  $F$  and $V $   a closed subvariety of $A$.  If  
  $V$ contains a dense subset of torsion points of $A$, then $V$ is the translate  of an abelian subvariety of $A$ by a torsion point.
\end{thm}

   \subsubsection{Idea  of the proof  of Theorem \ref{MM}}\label{1.1}

   We simply consider the case when $V$ does not contain any translate of a nontrivial abelian subvariety.
     Suppose that $A$ has good reduction at a place of $F$ unramified over a prime number $p$. 
     Let $[p]:A\to A$ be the morphism multiplication by $p$.  Let $\Lambda _n$ be  a suitable set of  reductions of   torsions in $[p^n]^{-1}(V)$, and $\Lambda _n^\Zar$ 
  its Zariski closure in  the base change to $\bar\BF_p$ of the  reduction of $A$. 
Use the $p$-adic perfectoid universal cover of  $A$ to lift    $\Lambda _n^\Zar$  to  $A $. 
A variant of Scholze's approximation lemma \cite{Sch12} shows that as $n$ get larger, the liftings are closer to $V$      (see Proposition \ref{techprop2}).  A  result of  Scanlon \cite{Sca0}   on  the Tate-Voloch conjecture for prime-to-$p$ torsions implies that the  prime-to-$p$ torsions  of these points are in $V$ for $n$ large enough
  (see Proposition \ref{techprop}). 
   Assume  that $\Lambda _n$ is infinite and we deduce a contradiction as follows.
   A result of   Poonen \cite{Poo} (see Theorem \ref{BT}) shows that the size of the set of prime-to-$p$ torsions in
$\Lambda _n^\Zar$ is not small.  
Then the liftings give   a   lower bound on the size of the set of prime-to-$p$ torsions in $V$  (see Proposition \ref{mainprop}).
Now consider the $l$-adic perfectoid space associated to $A$. By the same approach, we can repeatedly improve
 such     lower bounds. Finally
 we get a contradiction as $A$ is of finite dimensional.

\begin{rmk} 
  
  The proofs of Poonen's  result   and Scanlon's  result  are   independent of Theorem \ref{MM}.

           \end{rmk} 

         \subsection{Organization of the Paper}   
         The   preliminaries on adic spaces and perfectiod spaces are given in Section \ref{adic spaces and perfectiod spaces}. 
               We introduce the  perfectoid universal cover of an abelian scheme in  \ref{subThe Perfectoid universal cover of a formal abelian scheme}. The reader may skip these materials  and only come back for references.
We set up notations for 
       the proof of  Theorem \ref{MM} in \ref{Tilting and reduction}, then  prove Theorem \ref{MM} 
        in Section \ref{proof}. 
We introduce the   ordinary perfectoid Siegel space   and set up notations  
 for 
       the proof of  Theorem \ref{TVcor1} in Section \ref{The ordinary perfectoid Siegel space and Serre-Tate theory}. Then we prove Theorem \ref {TVcor1}  in Section \ref{pf{TVcor}}. 
        \subsection*{Acknowledgements}
       The author would like to thank Ye Tian and  Shouwu Zhang  for encouraging  him to study the theory perfectoid spaces,   and Shouwu Zhang for introducing him the work of Xie.   The author would also like to thank Ziyang Gao for carefully reading  versions of this work and giving useful suggestions,  as well as Vlad Serban, Xu Shen, Yunqing Tang and  Daxin Xu for their help. The author also thanks the anonymous referee for very careful reading and  helpful comments on the article.
         The author is grateful to the  Institute des Hautes \'Etudes Scientifiques  
       and  the Institutes for Advanced Studies at Tsinghua University 
       for their hospitality and support during the  preparation of part of this work. 
       Part of the result of this paper was reported in  the ``S\'eminaire Mathjeunes" at Paris in October 2016.
       The author would like to thank the organizers for the invitation.     
        
          \section{Adic spaces and perfectiod spaces}\label{adic spaces and perfectiod spaces}
 We briefly recall the theory of adic spaces due to Huber \cite{Hub}\cite{Hub93}\cite{Hub94}\cite{Hub96}, and   the generalization by Scholze-Weinstein \cite{SW}. Then we define tube neighborhoods in adic spaces  and  distance functions. Finally we recall the theory of perfectoid spaces of Scholze \cite{Sch12} and an approximation lemma due to Scholze.
 
    Let $K$ be a 
 non-archimedean field, i.e. a complete nondiscrete topological field  whose topology is inudced by a non-archimedean norm   $\|\cdot\|_K$ ($\|\cdot\|$ for short). Define
 $K^\circ=\{x\in K:\|x\|\leq 1\}$,   $K^\dbc=\{x\in K:\|x\|< 1\}$.        
    Let $\varpi\in K^\dbc-\{0\}$.

\subsection{Adic generic fibers of certain formal schemes   }\label{ADIC SPACES ASSOCIATED TO SCHEMES}
 
 \subsubsection{Adic spaces}

 Let $R$ be a complete Tate $K$-algebra, i.e. a complete topological $K$-algebra with a subring $R_0\subset R$ such that $\{aR_0:a\in K^\times\}$ forms a basis of open neighborhoods of 0. A subset of $R$ is called bounded if it is contained in a certain $aR_0$.
  Let $ R^\circ$ be the subring of power bounded elements, i.e. $x\in R^\circ$ if and only if the set of all powers of $x$ form a bounded subset of $R$.
 Let   $R^+\subset R^\circ$ be  an open integrally closed subring. Such a pair $ (R, R^+)$   is called   an affinoid $K$-algebra.
 Let $\Spa(R, R^+)$     
 be  the topological space  whose underlying set is the set of   equivalent classes  of continuous valuations  $|\cdot(x)|$ on $R$  such that     $|f(x)|\leq 1$  for every
  $f\in R^+$ and  
  topology  is generated by the 
  subsets of the form $$U(\frac{f_1,...,f_n}{g}):=\{x\in \Spa(R, R^+): \forall i, |f_i(x)|\leq |g(x)|\}$$ such that $(f_1,...,f_n)=R$. 
There is a natural presheaf    on $\Spa(R, R^+)$ (see \cite[p 519]{Hub94}).
    If this preasheaf is a sheaf, then the affinoid $K$-algebra  $(R,R^+)$ is called sheafy, and $\Spa(R,R^+)$ is called an affinoid adic space over $K$.
   
   \begin{asmp}\label{norm} If $K^\circ\subset R^+$, for every $x\in \Spa(R,R^+)$,
 we always choose a representative $|\cdot(x)|$ in the  equivalence class  of $x$  such that  
   $|f(x)|=\|f\|_K$  for every $f\in K$.  
\end{asmp}
  
 Define a category $(V)$ as in \cite[Definition 2.7]{Sch12}. Objects in $(V)$ are triples 
$ (\CX, \CO_\CX , \{|\cdot(x)|:x\in \CX\}) $ where  $(\CX,\CO_\CX)$ is a locally ringed 
topological space  whose  structure sheaf is a sheaf of complete topological $K$-algebras, and  $|\cdot(x)|$ is an equivalence class of continuous valuations  on  the stalk of $\CO_\CX$ at  $x$.  Morphisms  in $(V)$ are   morphisms of locally ringed topological spaces which are continuous $K$-algebra morphisms on  the structure sheaves, and compatible with the valuations on  the stalks in the obvious sense. 

 \begin{defn}  \label{adicdef}  
  An adic space $\CX$ over $K$ is an
  object   in $(V)$ which is  locally on $\CX$ an affinoid adic space over $K$.
        An adic space over $\Spa(K,K^\circ)$  is 
  an adic space over $K$ with a   morphism to $\Spa(K,K^\circ)$.   
   A morphism between  two  adic spaces  over $\Spa(K,K^\circ)$ is a morphism in $(V)$ compatible  the   morphisms to  $\Spa(K,K^\circ)$. 
    The set of  morphisms  $\Spa(K, K^\circ)\to \CX$    is denoted by $\CX(K,K^\circ)$.

   \end{defn}
 
 There is a natural inclusion 
$\CX(K,K^\circ)\incl \CX$
by mapping a morphism   $\Spa(K, K^\circ)\to \CX$ to its image. 
  We  always identify  $\CX(K,K^\circ)$ as a subset of $ \CX$ by this inclusion.

 \subsubsection{Adic generic fibers of certain formal schemes}

  A Tate $K$-algebra $R$  is called   of topologically finite type (tft for short) if $R$ is a quotient of $K   \langle T_1, T_2 ,..., T_n \rangle $. In particular, it is equipped with the $\varpi$-adic topology. Similarly define $K^\circ$-algebras of tft.
 By  \cite[5.2.6.Theorem 1]{BGR} and  \cite[Theorem 2.5]{Hub94},  if $R$ is of tft, then an affinoid $K$-algebra $  (R, R^+)$ is sheafy. 
Similar to the rigid analytic generic fibers of formal schemes over $K^\circ$ \cite[7.4]{Bos}, we naturally have  a  functor from the category of     formal schemes over $K^\circ$ locally of tft 
    to adic spaces over $\Spa(K,K^\circ)$ such that the image of $ \Spf A$ is   $  \Spa(A [\frac{1}{\varpi}],A ^c)$  where $A^c$ is the integral closure   of $A$ in $A[\frac{1}{\varpi}]$.   
  The  image  of a formal scheme under this functor  is  called its adic generic fiber.
 
       We are interested in certain infinite covers of abelian schemes and   Siegel moduli spaces. They are not of tft. We need to 
       generalize  the adic generic fiber functor.              In \cite{SW},  the category   of adic spaces over $\Spa(K,K^\circ)$ is enlarged  in a sheaf-theoretical way.    Moreover,       the adic generic fiber functor extends to 
 the category of formal schemes over $K^\circ$   locally admitting a finitely generated ideal  of definition. 
 
 For our purpose, we only need the following special case. Let $\fX$ be a formal $K^\circ$-scheme which is covered
              by affine open  formal subschemes  $ \{\Spf A_i:i\in I\} $,   where $I$ is an index set, such that each affinoid $K$-algebra 
               $ (A_i[\frac{1}{\varpi}],A_i^c)$ is sheafy.   
               Then the  adic generic fiber  $\CX$   of $\fX$ is an adic space over  $\Spa(K,K^\circ)$ in the sense of Definition \ref{adicdef}.
               Indeed, $\CX$  is the obtained by glueing the affinoid adic spaces  $\Spa (A_i[\frac{1}{\varpi}],A_i^c)$'s in the obvious way. 
                         We have an easy consequence.
 \begin{lem}\label{algpoints} Let $\CX$ be the adic generic fiber  $\fX$. Then  there is 
     a natural bijection  $\fX (K^\circ)\cong \CX(K, K^\circ).$ 
\end{lem}

                          \subsection{Tube neighborhoods and  distance functions}\label{tube neighborhoods in the adic generic fiber}  
 
                    \subsubsection{Tube neighborhoods}
  Let $\fX=\Spf B$, where $B$ is a flat $K^\circ$-algebra   of tft. 
  Let $\fZ$ be a closed formal subscheme defined by a    closed  ideal $ I$. 
  Let $\CX$ be the adic generic fiber of $\fX $.  Then $\CX=\Spa(R,R^+)$ where $R= B [\frac{1}{\varpi}]$ and $R^+ $ is the integral closure of $B $ in $R$.
     \begin{defn}\label{defnnbhd} 
        For $\ep\in K^\times ,$ the   $\ep$-neighborhood of $\fZ$  in $\CX$ is defined to be the   subset
                      $$\CZ_\ep: =\{x\in \CX:   |f (x)|\leq |\ep(x)| \mbox{ for every } f\in I \}.$$                                                                          \end{defn}
                       \begin{rmk} Note that $\CZ_\ep$ may not be open in $\CX$.
                        If $I$ is generated by $\{f_1,...,f_n\}$,  then  $\CZ_\ep=U(\frac{f_1,...,f_n,\ep}{\ep})$   is naturally  an open adic subspace  of $\CX$.         In fact,  for our applications, we only use this case.                 
                    \end{rmk}


                 Definition \ref{defnnbhd} immediately  implies the following lemmas.
\begin{lem} \label{intersect}                    Let $\fZ=  \bigcap\limits_{i=1}^m \fZ_i$, where each $\fZ_i$ is a  closed formal subschemes of $\fX$.
                      For $\ep\in K^\times$, let 
$\CZ_{i,\ep}$ be the 
$\ep $-neighborhood of $\fZ_i$. 
Then  $\CZ_\ep=   \bigcap\limits _{i=1}^m\CZ_{i,\ep}.$
 \end{lem}
 
 \begin{lem}\label{union} 
Let $\fZ= \fZ_1\bigcup  \fZ_2$, where $\fZ_1,\fZ_2$ are closed formal subschemes of $\fX$. 

(1) Then $\CZ_{1,\ep}\subset  \CZ_{ \ep}.$

(2)
Suppose that    there exists $\delta\in K^{\circ\circ}-\{0\}$ which vanishes on $\fZ_2$.
Then  
$\CZ_\ep\subset  \CZ_{1,\ep/\delta}.$
          \end{lem}

         Let $X$ be a  $K^\circ$-scheme  locally   of finite  type, and $\fX$   the $\varpi$-adic formal completion                     of    $ X $.              Let                  $\CX$ be the adic generic fiber of $\fX$.                     We also call $\CX$ the adic generic fiber of $X $.
                         Let $Z$ be a closed subscheme of   $X_K$.  We define tube neighborhoods of $Z$ in $\CX$ as follows  
                           (see         also \cite[Proposition 8.7]{Sch12}).

Suppose that $X $ is affine.    Let 
  $\fZ\subset \fX$ be the closed formal subscheme associated to the schematic closure of  $Z$.        \begin{defn} 
    For $\ep\in K^\times ,$ the   $\ep$-neighborhood of $Z$  in $\CX$ is defined to be the        $\ep$-neighborhood of $\fZ$  in $\CX$.                                                                \end{defn}
                                        
                      \begin{rmk}     If the schematic closure of $Z$
     has empty special fiber,
     then   $\CZ_\ep$ is empty.
                    \end{rmk}                     
                
                                 To  define tube neighborhoods   in general, we need to glue  affinoid pieces.
               We consider the following relative situation.    
            Let   $ Y $ be another   affine  $K^\circ$-scheme     of finite type,  and $\Phi:Y\to X$  a $K^\circ$-morphism. Let $W$ be the preimage of $Z$ which is a closed subscheme of $Y_K$, and $\CW_\ep$   its $\ep$-neighborhood.  
            By the functoriality of formal completion and taking adic generic fibers, we have an  induced morphism        $\Psi:\CY\to \CX$.
  From the fact that schematic image is compatible with flat base change (see   \cite[2.5, Proposition 2]{BLR}), we easily deduce the following lemma.
   \begin{lem}\label{pullback1}If  $\Phi:Y\to X$ is flat, then $\Psi^{-1}(\CZ_{\ep})=\CW_{ \ep}$. 
 In particular,  if $Y\subset X$ is an open $K^\circ$-subscheme,  $\CW_\ep=\CZ_\ep \bigcap \CY$ under the natural inclusion $\CY\incl \CX$.
  \end{lem}

 Now we turn to the general case.   
 Let $X$ be an $K^\circ$-scheme locally of finite type. For an open subscheme $U\subset X$, let  $Z_U$ be the restriction of $Z$ to $U$.  
 Let $S=\{U_i:i\in I\}$ be an affine open cover of $X$, where $I$ is an index set and each $U_i $ is of finite type over $K^\circ$. 
   Let $\CZ_{U_i,\ep}$ be the $\ep$-neighborhood  of $Z_{U_i}$  in the adic generic fiber $\CU_i$ of $U_i$.  Note that each $\CU_i$  is naturally an open adic subspace of $\CX$.
 \begin{defn} \label{globalnbhd}Define the $\ep$-neighborhood of $Z$ in $\CX$ by
 $\CZ_\ep:=\bigcup\limits_{U\in S}\CZ_{U,\ep}.$

  \end{defn} 
   
As a corollary of  Lemma \ref{pullback1}, this definition is independent of  the choice of  the cover $\CU$.

\subsubsection{Distance functions} \label{The distance function} 

 Let   $U$ be an affine open subset  of $X$ which is flat over $K^\circ$. Let $I $ be an ideal of the coordinate ring of $U$. 
For $x\in U(K) $,  
define  $d_U(x,I):=\sup\{\|f(x)\|:f\in I\}.$ 
   Let $\CI$ be the ideal sheaf of the schematic closure of $Z$ in $X$.
 
   Assume that $X$ is of finite type over $K^\circ$.
 Let     $\CU:=\{U_1,...,U_n\}$ be a finite  affine open cover of $ X $ such that each $U_i$ is flat over $K^\circ$.
            For $x\in X(K)$, define $d^\CU(x,Z)$ to be the maximum of $ d_{U_i}(x,I)$ over all $i$'s such that $x\in U_i$. 


    
Let 
   $x^\circ\in X(K^\circ)$ and $x$   the generic point of $x^\circ$.  Regard $x$ as a point in $\CX(K,K^\circ)$ via Lemma \ref{algpoints}.
    Let   $U$ be an affine open subset  of $X$   flat over $K^\circ$ such that $x^\circ\in U(K^\circ)$.   We have a  tautological relation between the distance function   and tube neighborhoods. 
 
           \begin{lem} \label{distep}
       Let $\ep\in K^\times$. 
Then  $x\in \CZ_{U,\ep}$ if and only if     $d_U (x,\CI(U))\leq \|\ep\|$.

                           \end{lem} 
                 By  Lemma   \ref{pullback1},    the number  $d_U (x,\CI(U))$   does not depend on the choice of $U$.
     Define $$d(x,Z):= d_U (x,\CI(U))$$
     Then  $d(x,Z) = d^\CU(x,Z) $ for every finite  affine open cover  $\CU $ of $X$. 
Our distance function  coincides with the one  in the end of \cite[Section 1]{Sca0}, which is defined globally.

A finite extension  of $K$ has a natural structure of a non-archimedean field (see \cite{BGR}). Let $\bar K$ be an algebraic closure of $K$.
    The above discussion is naturally generalized to      $x \in X(\bar K)$ and $Z\subset X_{\bar K}$.

\subsubsection{Tate-Voloch type  sets}\label{TVset} Let $X$ be of finite type over $K^\circ$.
 \begin{defn}\label{TVtype}Fix an arbitrary   finite  affine open cover  $\CU$ of $ X $ by subschemes flat over $K^\circ$.
A  set 
 $T\subset X(\bar K)$    is of Tate-Voloch type if for every closed subscheme $Z$ of $X_{\bar K}$, there exists a constant $c>0$ such that for every 
    $x\in T$, if $d^\CU(x,Z)\leq c$, then $x\in Z(\bar K)$.   
    \end{defn}
  \begin{rmk}  Is there always a set of Tate-Voloch type? Let $C\subset X$ be irreducible and flat over $K^\circ$ of relative dimension 1. Choose one point in each residue disk in $C$. 
  Easy to check that this set of points of $X$ is   of Tate-Voloch type. Moreover,  we can choose  points in   residue disks in $C$ chose degrees are unbounded. 
The following questions are  more meaningful.   Is there always a   Tate-Voloch type set  which is  Zariski dense   in  $X$?
Can  the points   in this set have   unbounded
  the degrees over $K$? 
     Indeed, the   Tate-Voloch type sets  in Theorem \ref{TVcor1}   and in  the results of  Habegger, Scanlon   and 
Xie  give positive answers to these two questions.
\end{rmk}

     Let $Y$ be a  $K^\circ$-scheme  of finite type,
and $\pi:Y\to X$   a  finite  schematically dominant  morphism.
\begin{lem}\label{T'T}Let $T\subset X(\bar K)$ be of Tate-Voloch type and   $T'=\pi^{-1}(T)\subset Y(\bar K)$. 
Then   $T'$ is of Tate-Voloch type.
\end{lem}

\begin{proof}We may assume that $Y=\Spec B$ and $X=\Spec A$  where $A$ is a subring of $B$. 
 Let  $L$ be a finite extension of $K$.
 Let $Z'$ be  a closed subscheme  of $Y_{L}$.  We need to show that $d(x', Z')$ has a positive lower bound  for $x'\in T'-Z'({\bar K})$.  
 Define the dimension of $Z'$ to be the maximal dimension of the irreducible components of $Z'$. We allow $Z'$ to be empty, in which case we
define its dimension to be $-1$. We do induction on the dimension of $Z'$.
Then the dimension $-1$ case is trivial.  Now we consider the general case with the hypothesis that the lemma holds for all lower dimensions. 

Suppose such a lower bound does not exists, then there exists a sequence of $x_n' \in T'- Z'({\bar K})$
 such that  $d(x_n',Z
')\to  0$ as $n\to \infty$. We will find a contradiction.
Let $Z$ be the schematic image of $Z'$ by $\pi$, $x_n=\pi(x_n')$.  
Let the   schematic  closure  
 of $Z$ in $X_{L^\circ }$(resp. $Z'$ in $Y_{L^\circ }$) be  defined  by an ideal $J\subset A\otimes{L^\circ }$ (resp.  $I\subset B\otimes{L^\circ }$).
 Then 
 $I\otimes L \supset 
 J B\otimes L$.
Since      $J  B\otimes{L^\circ }$ is finitely generated, 
 there exists a positive integer $r$ such that    $I \supset 
 \varpi^rJ  B\otimes{L^\circ }$. Thus   $d(x_n,Z
 )\to  0$ as $n\to \infty$.
 Since $T$ is of Tate-Voloch type, $ x_n\in Z(\bar K)$ for $n$ large enough. We may assume that every $ x_n\in Z(\bar K)$.
Since $x_n'\not\in Z'$, $\pi^{-1}(Z)=Z'\bigcap Z_1$ where $Z_1$ is a closed subscheme of $Y_L$ not containing $Z'$ but containing all
$x_n'$. Claim: $d(x_n', Z'\bigcap Z_1)\to 0$  as $n\to \infty$.
 This contradicts  the induction hypothesis. Thus  $d(x', Z')$ has a positive lower bound  for $x'\in T'-Z'(K)$. 
Now we prove the claim.
 Let the   schematic  closure   of 
 $Z_1$ in $Y_{L^\circ}$ be defined by an ideal $I_1\subset B \otimes{L^\circ }$. 
 Then the schematic closure of 
  $Z'\bigcap Z_1$ is defined by the following ideal of $ B \otimes{L^\circ }$: $$I_2:=(I_1\otimes{L  }+I'\otimes{L })  \bigcap  B \otimes{L^\circ }=(I_1 +I')\otimes{L }  \bigcap  B \otimes{L^\circ },$$ which is finitely generated.
Thus there exists a positive integer $s$ such that $(I_1+I')  \supset \varpi ^s I_2$. 
Now the claim follows from that  $d(x_n',Z
')\to  0$  and $x_n'\in Z_1$.
  \end{proof}

 \subsection{Perfectoid spaces}  \subsubsection{Two perfectoid fields}\label{perftheory}
Instead of recalling the definition of perfectoid fields (see \cite[Definition 3.1]{Sch12}),
we consider two examples and use them through out this paper.    

  Let  $k=\bar \BF_p$,  $W=W(k)$ the ring of Witt vectors, and  $L=W[\frac{1}{p}]$. For  each integer $n\geq 0$, let $\mu_{p^n}$ be a primitive $p^n$-th root of unity in $\bar L$ such   that $\mu_{p^{n+1}}^p=\mu_{p^n}$.
  Let $$L^\cycl:=\bigcup\limits_{n=1}^\infty L(\mu_{p^n}).$$ 
  Let $\varpi=\mu_p-1$, 
  and $K$
  the $\varpi$-adic completion of $L^\cycl$. 
Then $K$ is a perfectoid field  in the sense that   
$$K^\circ/\varpi\to K^\circ/\varpi ,\ x\mapsto x^p$$ is surjective (see \cite[Definition 3.1]{Sch12}). Let $$K^\flat= k((t^{1/p^\infty})) $$ be the $t$-adic completion  of $ \bigcup\limits_{n=1}^\infty k((t))(t^{1/p^n})$.
  Then   $ K^\flat$  is a perfectoid field.
 Let   $\varpi^\flat=t^{1/p}$. Equip $ {K^\flat}$ with  the nonarhimedean norm  $\|\cdot\|_{K^\flat}$ such that   $\|\varpi^\flat\|_{K^\flat}=\|\varpi\|_{K}$.
 Consider the  morphism 
  \begin{equation}K^\circ/\varpi\to K^\fcc/\varpi^\flat,\ \mu_{p^n}-1\mapsto t^{1/p^n}\label{2.3.2eq}.\end{equation}
This morphism is well-defined since $$(\mu_{p^n}-1)^{p^m}\cong \mu_{p^{n-m}}-1(\mod \varpi)$$ for $m<n$.
Easy to check this morphism  is an isomorphism. We call  $ K^\flat$   the tilt of $K$.   

\subsubsection{Perfectoid spaces} 
The most  important property of a perfectoid  $K$-algebra $R$ is  that $$R^\circ/\varpi\to R^\circ/\varpi ,\ x\mapsto x^p$$ is surjective (see
 \cite[Definition 5.1]{Sch12}). 
  An affinoid $K$-algebra $(R,R^+)$ is called perfectoid if $R$ is perfectoid.
 By  \cite[Theorem 6.3]{Sch12}, an affinoid $K$-algebra $(R,R^+)$  is sheafy.
Define a perfectoid space   over $K$ to be  an adic space over $K$ 
  locally isomorphic to $\Spa(R,R^+)$, where $(R,R^+)$
is a perfectoid affinoid $K$-algebra. 

By   \cite[Theorem 5.2]{Sch12}, there is an equivalence  between the categories of  perfectoid  $K$-algebras and perfectoid  $K^\flat$-algebras.  
By \cite[Lemma 6.2]{Sch12},  and  \cite[Proposition 6.17]{Sch12}, this category equivalence induces an equivalence between
between the categories of perfectoid  affinoid $K$-algebras and perfectoid  affinoid $K^\flat$-algebras, as well as an equivalence between
between the categories of perfectoid spaces  over $K$ and perfectoid spaces  over $K^\flat$.
  
 The image of an object or a morphism in the category of perfectoid  $K$-algebras, perfectoid  affinoid $K$-algebras,  or 
perfectoid spaces   over $K$ is called its tilt.
  
   \subsubsection{Two important maps $\sharp$ and $\rho$}\label{srho}
   Let $R$ be perfectoid  $K$-algebra  and $R^{\flat}$ its tilt.
    By \cite[Proposition 5.17]{Sch12}, there is a     multiplicative homeomorphism 
    $R^{\flat}\cong   \vpl\limits_{x\mapsto x^p} R .$ 
    Denote the projection to the first component by $$R^\flat\to R, \ f\mapsto f^\sharp.$$ 
    Let $(R,R^+)$ be perfectoid  affinoid $K$-algebra and $(R^\flat,R^\fpl)$ its tilt.
For $x\in \Spa(R,R^+), $ let $\rho(x)\in  \Spa(R^\flat,R^\fpl)$  be the valuation
   $|f(\rho(x))|=|f^\sharp(x)|$ for $f\in R^\flat$. This defines a map between sets
$$\rho: \Spa(R,R^+)\mapsto \Spa(R^\flat,R^\fpl).$$   
Note that  $  \Spa(R^\flat,R^\fpl)$ is the tilt of $  \Spa(R,R^+) $. The definition of $\rho$ glues and we have a 
map    $$\rho_\CX:|\CX|\cong|\CX^\flat|$$ between the underlying sets of a perfectoid space $\CX$ over $K$ and its tilt $\CX^\flat$.

  \begin{lem}\label{tiltmor}
  (1)  Let $\phi:R\to S$ be a morphism between perfectoid $K$-algebras, and 
   $\phi^\flat:R^\flat\to S^\flat$ its tilt. Then for every  $f\in R^\flat$, we have $\phi^\flat(f)^\sharp=\phi(f^\sharp).$
     
  (2) Let $\Phi:\CX\to\CY$ be a morphism between perfectoid spaces over $K$ and $\Phi^\flat$ its tilt. Then as maps between topological spaces, we have $$\rho_\CY\circ \Phi=\Phi^\flat\circ \rho_{\CX}$$ 
        \end{lem}
        \begin{proof} (1) follows from    the definition of  the $\sharp$-map  and  \cite[Theorem 5.2]{Sch12}. (2) follows from (1).        \end{proof}
        By (2), the restriction of    $\rho_\CX $ to $\CX(K,K^\circ)$ 
gives  the functorial  bijection $\CX(K,K^\circ)\cong \CX^\flat(K^\flat,K^{\flat\circ})$,
which we also denote by  $ \rho_\CX$.        
In the next two paragraphs, we compute $ \rho_\CX$ in two cases.
            
       \subsubsection{Tilting and reduction}\label{subTilting and reduction} 
Let  $(R,R^+)$  be   a perfectoid affinoid $K $-algebra  and $(R^\flat,R^\fpl)$
 its tilt. Suppose there exists    a flat $W$-algebra  $S$  such that 
  \begin{itemize}
  \item[1] $R^+ $ is  the $\varpi$-adic completion of $S\otimes_W K^\circ$,
  \item[2]  $R^\fpl  $   is the $\varpi^\flat$-adic completion of $S_k  \otimes_k K^{\flat\circ}$. 
  
\end{itemize}  
               Let $\phi:S\to W$ be a $W$-algebra morphism, $\phi_k:S_ k\to k$ be its base change. Then  $\phi $ induces a map $\psi:R^+ \to K^\circ$  which further induces a  point  $x$ of $\Spa(R,R^+)$. Similarly,  
     $\phi_k $ induces  a map $\psi':R^\fpl \to K^\fcc$  which further induces  a   point  $x'$  of $\Spa(R^\flat,R^\fpl)$.
Then $\psi/\varpi
= \psi'/\varpi^\flat$  under the isomorphism $R^+/\varpi\cong R^\fpl/\varpi^\flat.$ By   \cite[Theorem 5.2]{Sch12},  $\phi'$ is the tilt of $\phi$  
and thus we have the following lemma.      \begin{lem}\label{tiltred2}     
We have  $\rho_{\Spa(R,R^+)}(x)=x'$. 
\end{lem}
\subsubsection{An example: the perfectoid closed unit disc}\label{CGperf}

Let $R=K\pair{T^{1/p^\infty},T^{-1/p^\infty}}$, the $\varpi^\flat$-adic completion of  $\bigcup\limits_{r\in\BZ_{\geq 0}} K [ T^{1/p^r}, T^{-1/p^r}].$   Then $R$ is perfectoid.  The tilt $R^\flat$ of $R$ is  $K^\flat\pair{T^{1/p^\infty},T^{-1/p^\infty}}$.
Let  $ \CG^\perf:=\Spa (R,R^\circ) $. Then   $\CG^\perf$  is a    perfectoid space over $\Spa(K,K^\circ)$, and 
 $ \CG^{\perf,\flat}:=\Spa(R^\flat, R^\fcc)$ is its tilt.
 
      Let $c \in \BZ_p$, and $m\in  \BZ_{\geq 0}$. 
The $K^\circ$-morphism $R ^\circ\to K^\circ$ defined by $$T ^{1/p^n}\to \mu_{p^{m+n}}^c$$
  gives a point $x\in \CG^\perf(K,K^\circ)$.
The $K^\fcc$-morphism $R^\perf \to K^\flat$ defined by $$T ^{1/p^n}\to (1+t^{1/p^{m+n}})^c$$
   gives a point $x'\in \CG^{\perf,\flat}(K^\flat,K^\fcc)$.  
   The following lemma  follows  from \eqref{2.3.2eq} and  \cite[Theorem 5.2]{Sch12}.
   \begin{lem} \label{2316} We have $\rho_{\CG^\perf}(x)=x'$.  
   
   \end{lem}
   
Similar result holds for 
    $\CG^{l,\perf}=\Spa(R,R^\circ)$ where $$R=K\pair{T_1^{1/p^\infty},T_1^{-1/p^\infty},...,T_l^{1/p^\infty},T_l^{-1/p^\infty}},$$ 
and its tilt    $\CG^{l,\perf\flat}=\Spa(R^\flat,R^\fcc)$ where $$R^\flat=K^\flat\pair{T_1^{1/p^\infty},T_1^{-1/p^\infty},...,T_l^{1/p^\infty},T_l^{-1/p^\infty}}.$$

  \subsection{A variant of Scholze's approximation lemma}\label{some estimates on  perfectoid spaces}
  The perfectoid fields $K$, $K^\flat$ and related notations are as in \ref{perftheory}.
  Let $(R,R^+)$ be a perfectoid  affinoid $(K,K^\circ)$-algebra with tilt $(R^\flat,R^\fpl)$.
Let $\CX=\Spa(R,R^+)$ with tilt
  $\CX^\flat=\Spa(R^\flat,R^\fpl)$. 
  For $f,g\in R$, define $|f(x)-g (x)| $   to be $|(f-g)(x)|$.
The following approximation lemma plays an important role in Scholze's work \cite{Sch12}.
 
\begin{lem}[{\cite[Corollary 6.7 (1)]{Sch12}}]\label{Corollary 6.7. (1)} Let $f\in R^+$.
Then   for every $ c\geq 0$, there exists 
  $g\in R^\fpl$   such that  for every $x\in \CX$, we have
  \begin{equation}|f(x)-g^\sharp(x)|  \leq\|\varpi\|^{\frac{1}{p}} \max\{|f(x)|,\|\varpi\|^c\} =\|\varpi\|^{\frac{1}{p}} \max\{|g^\sharp(x)|,\|\varpi\|^c\}.\label{2.1}\end{equation}    \end{lem}
  Here   the map    $\sharp$ is as in \ref{srho} (i.e. $|g(\rho(x))|=|g^\sharp(x)|$), and  we use $\|\cdot\|$ to denote $\|\cdot\|_K$.

  Recall that   $k=\bar \BF_p$.
Assume that there exists  
 a  $k$-algebra $S$, such that  $ R^{\flat+}$ is the $\varpi^\flat$-adic completion of $S \otimes K^{\flat\circ}$.
Then we have natural maps\begin{equation*}\Hom_k(S,k)\incl \Hom_ {K^{\flat\circ}}( S \otimes K^{\flat\circ},K^{\flat\circ})  \cong \CX^\flat(K^\flat,K^{\flat\circ}) .\end{equation*} 
Thus we regard $  (\Spec S)(k) $ as a subset of $ \CX^\flat$.
   
 \begin{lem}\label{improveCorollary 6.7. (1)} Continue to use the notations  in Lemma \ref{Corollary 6.7. (1)}. Assume  that
  $ c\in \BZ[\frac{1}{p}] $. There exists 
  a finite sum   $$g_c= \sum_{\substack{i\in \BZ[\frac{1}{p}]_{\geq 0},\\i< \frac{1}{p}+c}}g_{c,i}  \cdot (\varpi^\flat)^i $$ with $g_{c,i}\in S$ and only finitely many $g_{c,i}\neq 0$,
     such that    \begin{equation}g-g_c\in (\varpi^\flat)^{\frac{1}{p}+c}R^{\flat+}.\label{modi}
  \end{equation} 

    \end{lem}
        \begin{proof}    There exists a \textit{finite  sum}
   $g'=\sum s_j a_j \in   S \otimes K^{\flat\circ}$, where $s_j\in S$ and $a_j\in K^\fcc$, such that 
    $g-g'\in (\varpi^\flat)^{\frac{1}{p}+c}R^{\flat+}.$ 
 Claim:  let $a\in K^{\flat\circ}$, then there exists a positive integer $N$   such that
  $$a-\sum_{\substack{h\in (\frac{\BZ}{ p^N})_{\geq 0},\\h< \frac{1}{p}+c }} \alpha_h \cdot (\varpi^\flat)^h \in (\varpi^\flat)^{\frac{1}{p}+c}K^{\fcc}$$   for certain $\alpha_h\in  k$. Indeed,  the claim  follows from that $K^{\fcc}$ is the $\varpi^\flat$-adic completion of $ \bigcup\limits_{n=1}^\infty k[[t]][  (\varpi^\flat)^{1/p^n}]$.
Note that  $\{h\in (\frac{\BZ}{ p^N})_{\geq 0}, h< \frac{1}{p}+c \}$ is finite set.
So there exists a  finite sum  $$g_c=\sum_{\substack{i\in \BZ[\frac{1}{p}]_{\geq 0},\\i< \frac{1}{p}+c}}g_{c,i}  \cdot (\varpi^\flat)^i $$ with $g_{c,i}\in S$ 
 such that 
$g'-g_c\in (\varpi^\flat)^{\frac{1}{p}+c}R^{\flat+}.$  
Then
$g-g_c\in (\varpi^\flat)^{\frac{1}{p}+c}R^{\flat+}.$            \end{proof}

        \begin{lem} \label{g=0}Let $g_c$ be as in   Lemma \ref{improveCorollary 6.7. (1)}   and  $x\in   (\Spec S)(k)$.
     Regarding $x\in \CX^\flat(K^\flat,K^\fcc)$ via the inclusion above.  If         $|g_c(x)| \leq  \|\varpi\|^{\frac{1}{p}+c}$, then $g_{c,i}(x)=0$ for all $i$. 
 
         \end{lem}
        
        \begin{proof} Since $x\in   (\Spec S)(k)$,  if $g_{c,i}(x)\neq 0$, then $|g_{c,i}(x)|=1$.  Let  $i_0< \frac{1}{p}+c$  be the minimal $i$ 
 such  that $|g_{c,i}(x)|=1$.
        Then $ |g_c(x)|=\|\varpi^\flat\|_{K^\flat}^{i_0}>   \|\varpi\|^{\frac{1}{p}+c}$, a contradiction.
        \end{proof}
     \subsubsection{Profinite setting}   Impose the following assumption.
    \begin{asmp} \label{asmp4}
There are  $k$-algebras
  $S_0\subset S_1\subset...$ such that $S=\bigcup S_n$.
  \end{asmp}
  Let $\CX_n $  is the adic generic fiber of $\Spec S_n\otimes K^\fcc$. Then we have a natural morphism
  $$\pi_n:\CX^\flat\to \CX_n.$$ We also use $\pi_n$ to denote the morphism $(\Spec S)(k)\to (\Spec S_n)(k)$. We have natural maps
   $$ (\Spec S_n)(k) \incl \Hom_{K^\fcc}(S_n\otimes K^\fcc, K^{\flat\circ})\cong \CX_n (K^\flat,K^{\flat\circ})  $$ by which 
     we regard $(\Spec S_n)(k)$  as a subset of $\CX_n$. 
For each $n$, let $\Lambda_n\subset  (\Spec S_n)(k) $ be a set of $k$-points, and    $\Lambda^\Zar_n$    the Zariski  closure of $\Lambda _n$ in $\Spec S_n$. 
We have the following maps and inclusions between sets:
$$|\CX|\xrightarrow{\rho}|\CX^\flat|\xrightarrow{\pi_n}|\CX_n| \supset \Lambda _n^\Zar  (k)   \supset \Lambda_n. $$
where $\rho$ is as in \ref{srho}.

 Let $f\in  R^+$, and $\Xi:=\{x\in  \CX :|f(x )|=0\}$. We have the following variant of  Lemma \ref {Corollary 6.7. (1)}.
   \begin{prop}\label{techprop2} Assume that $ \Lambda_n\subset \pi_n ( \rho( \Xi))$ for each $n$.
  Then for each $\ep\in K^\times$, there exists a positive integer $n$ such that   
$|f (x)|\leq \|\ep\|_K$ for every $x\in ( \pi_n\circ \rho)^{-1}\left(\Lambda _n^\Zar(k) \right)$.
 \end{prop}

        \begin{proof}  
 Choose $c\in \BZ _{\geq 0}$ large enough such that $ \|\varpi\|_K^{\frac{1}{p}+c}\leq \|\ep\|_K $,   choose $g$ as in Lemma \ref{Corollary 6.7. (1)} and choose a finite sum  $$ g_c=\sum\limits_{\substack{i\in \BZ[\frac{1}{p}]_{\geq 0},\\i< \frac{1}{p}+c}}g_{c,i}   \cdot (\varpi^\flat)^i $$
  as in Lemma \ref{improveCorollary 6.7. (1)} where $g_{c,i}\in S$ for all $i$.
        There  exists a positive integer  $n(c)$ such that $g_{c,i}\in S_{n(c)}$ for all $i$ by the finiteness of the sum.    By the assumption,  every element   $x\in  \Lambda_{n(c)}$   can be  written as $\pi_{n(c)} \circ \rho (y)$ where $y\in \Xi$. By \eqref{2.1} and  \eqref{modi},   $|g_c(\rho(y) )|\leq  \|\varpi\|^{\frac{1}{p}+c}$. 
                Then by Lemma \ref{g=0} and that   $\rho(y)\in (\Spec S)(k)$, $g_{c,i}(\rho(y))=0$. Since
      $g_{c,i}\in S_{n(c)}$,  $g_{c,i}(x)=0$.
     Thus  $g_{c,i}$ lies in the ideal defining $\Lambda_{n(c)}^\Zar$.
      So   $g_{c,i}(x)=0$, and thus $g_c(x)=0$, for every $x\in \Lambda _{n(c)}^\Zar(k) $. By \eqref{2.1} and   \eqref{modi},  for every  $x\in \Lambda _{n(c)}^\Zar(k) $, we have $$|f\left(\rho^{-1}\left(\pi_{n(c)} ^{-1}(x)\right)\right)|\leq  \|\varpi\|^{\frac{1}{p}+c}\leq\| \ep\|.$$ \end{proof}

\section{Perfectoid universal cover of an abelian scheme}\label{The perfectoid universal cover of a formal abelian scheme}
Let  $K$  be the  perfectoid field in \ref{perftheory} and  $K^\flat$ its tilt. Let $\fA$ be a formal abelian scheme  over $K^\circ$.
We   first recall the perfectoid universal cover of $\fA$ and its tilt constructed in  \cite[Lemme A.16]{PB}.
Then we study the relation between tilting and reduction.
\subsection{Perfectoid universal cover of an abelian scheme}\label{subThe Perfectoid universal cover of a formal abelian scheme}

 Let    $\fA'$  be a formal abelian  scheme  over $\Spf K^{\flat\circ} $. Assume that there is an isomorphism  \begin{equation}\fA\otimes   K^\circ/\varpi\cong \fA'\otimes   K^{\flat\circ}/\varpi^\flat\label{modpicond}\end{equation}    of abelian schemes over $K^\circ/\varpi\cong K^{\flat\circ}/\varpi^\flat$.
  Let 
\begin{equation*}\tilde\fA:=\vpl\limits_{[p]} \fA,\   \tilde\fA':=\vpl\limits_{[p]} \fA'.\end{equation*} Here the transition maps  $[p]$ are the morphism   multiplication by $p$ and 
 inverse limits 
 exist in the categories of   $\varpi$-adic and $\varpi^\flat$-adic  formal schemes  (see  \cite[Lemme A.15]{PB}).      Index the inverse systems  by $\BZ_{\geq 0}$.  Let $ \Spf R_0^+\subset \fA$ be an affine open formal subscheme. 
  Let $R_i^+$ be the coordinate ring of $([p]^{i})^{-1} \Spf R_0^+ $, in other words, $ \Spf R_i^+=([p]^{i})^{-1} \Spf R_0^+ $.
 Let $R_i=R_i^+[\frac{1}{\varpi}]$,
  then $R_i^+ $ is integrally closed in $R_i$.
  Let 
$R^+ $ be the  $\varpi$-adic completion of $\bigcup\limits_{i=0}^\infty R_i^+$, 
$ R =R^+[\frac{1}{\varpi}] .$ 
Let $ \Spf R_0'^+\subset \fA'$ be an affine open formal subscheme such that the  restriction of \eqref{modpicond} to $ \Spf R_0^+\otimes   K^\circ/\varpi$  is an isomorphism to $ \Spf R_0'^+ \otimes   K^{\flat\circ}/\varpi^\flat$.
We similarly define $R_i'^+$, $R'^+$ and $R'$.

\begin{lem}   [{\cite[Lemme A.16]{PB}}]\label{A.16} 
 The affinoid $ K^\flat $-algebra $(R',R'^+)$ is perfectoid. So is  $(R,R^+)$.
 Moreover,  $(R',R'^+)$ is the tilt of $(R,R^+).$
   \end{lem}
   Thus 
   the  adic generic fiber  $\CA^\perf$ (resp. $\CA'^{\perf}$)  of  $\tilde\fA$ (resp. $\tilde\fA'$) is a perfectoid space. 
        Moreover,   $ \CA'^{\perf}$ is the tilt of $ \CA^\perf$. Thus we use $ \CA^{\perf\flat}$ to denote $ \CA'^{\perf}$.
     We call $\CA^\perf$ (resp. $\CA^{\perf\flat}$)  the  perfectoid universal cover  
of $ \fA$ (resp. $ \fA'$).    
 By Lemma \ref{algpoints}, there  are natural bijections
$$\tilde \fA(K^\circ)\cong \CA^\perf(K,K^\circ),\ \tilde \fA'(K^\fcc)\cong \CA^{\perf\flat}(K^\flat,K^\fcc).$$
  Let $\CA$ (resp.  $\CA' $) be the adic generic fiber of $\fA$ (resp. $\fA'$). 
By Lemma \ref{algpoints}, we have   natural   bijections  $$\fA (K^\circ)\cong \CA(K,K^\circ),\ \fA' (K^\fcc)\cong \CA'(K^\flat,K^\fcc).$$
\begin{defn} \label{adicgroup}
  The group structures on $\CA(K, K^\circ)$,  $\CA^\perf(K, K^\circ)$, $\CA'(K^\flat,K^\fcc)$,and  $\CA^{\perf\flat}(K^\flat,K^\fcc)$ are defined  to be the ones
 induced from  the natural bijections above. 
 \end{defn}

  By the functoriality of taking adic generic fibers, we have  morphisms  $$\pi_n:\CA^\perf \to \CA,\ \pi_n':\CA^{\perf\flat} \to \CA' $$
  for $n\in \BZ_{\geq 0},$ and 
  morphisms
$$[p]:\CA\to \CA, \ [p]:\CA'\to \CA'.$$
Consider
   the following commutative diagram    
 \begin{equation}  
    \xymatrix{
   	\tilde\fA(K^\circ)   \ar[r]^{\cong } \ar[d]^{\cong}  &\vpl\limits_{[p]} \fA (K^{ \circ})  \ar[d]^{\cong} \\
\CA^\perf(K,K^\circ) \ar[r]^{   }    & \vpl\limits_{[p]}   \CA(K,K^\circ)}\label{faca}
\end{equation}
where the bottom map is given by $\pi_n$'s. We immediately have the following lemma.
\begin{lem}\label{faca'}
    The bottom map in \eqref{faca} is  a group isomorphism.
    \end{lem}
      \begin{rmk}  Indeed,
$\CA^\perf$ serves as certain ``limit" of the inverse system $\vpl \CA$
in the sense of \cite[Definition 2.4.1]{SW} by \cite[Proposition 2.4.2]{SW}.    Then Lemma \ref{faca'}   also follows from \cite[Proposition 2.4.5]{SW}.  
  \end{rmk}
  Now we study torsion points in the inverse limit. 
  We set up some group theoretical convention once for all.
  Let $G$ be an abelian group.    We  denote by  $G[n]$ the subgroup   of  elements  of  orders dividing $n$ and
  by $G_{\tor}$   the subgroup of  orsion elements. For a prime $p$,
we use   $G[p^\infty]$ to denote the subgroup   of $p$-primary torsion points, and  
 $G_{p'-\tor}$ to denote the subgroup   of prime-to-$p$ torsion points. If $H$ is a \textit{subset} of $G$,    $H_\tor$ and  $H_{p'-\tor}$ to denote the subset      $H\bigcap G_{\tor}$ 
 and $H\bigcap G_{p'-\tor}$ when both the definitions of $H$ and $G$ are clear from the context.  
  The following lemma is elementary. 
 \begin{lem}\label{gpth}   Let $G$ be an abelian group, then
 $$  ( \vpl\limits_{[p]} G) _{p'-\tor}\cong\vpl\limits_{[p]} G _{p'-\tor}.$$  
 
\end{lem}

         \begin{lem}   \label{4isoms}      There are group isomorphisms  $$ \CA^\perf(K,K^\circ)_{p'-\tor} \cong \vpl\limits_{[p]} \CA(K,K^\circ)_{p'-\tor}  \cong \CA(K,K^\circ)_{p'-\tor} $$
         where the second isomorphism is  the restriction of  $\pi_n$.
         Similar result holds for $\CA'$ and $\CA^{\perf\flat}$.
         \end{lem} 
  
 \begin{proof} The first isomorphism is from Lemma \ref{faca'} and  \ref{gpth}.
 Since  $\CA(K,K^\circ)[n]\cong \fA(K^\circ)[n]$ is a finite group, 
$[p]$ is an isomorphism on $\CA(K,K^\circ)[n]$ for every natural number $n$ coprime to $p$. The second isomorphism follows.
  \end{proof}

      \begin{prop} \label{perfgroup} 
       The functorial bijection $$\rho=\rho_{ \CA^\perf}: \CA^\perf(K,K^\circ)\cong  \CA^{\perf\flat}(K^{ \flat },K^{ \flat\circ})$$  (see  \ref{srho})  is a group isomorphism.
          \end{prop}

   \begin{proof} 
   We only show the
  compatibility of $\rho$ with the multiplication maps, i.e. we show that the following  diagram is commutative:
 $$
    \xymatrix{
   \CA^\perf(K,K^\circ)\times \CA^\perf(K,K^\circ) \ar[r]^{\rho \times  \rho  } \ar[d]^{}  & \CA^{\perf\flat} (K^\flat,K^\fcc)\times \CA^{\perf\flat} (K^\flat,K^\fcc)  \ar[d]^{} \\
    \CA^\perf(K,K^\circ)\ar[r]^{ \rho }    & \CA^{\perf\flat} (K^\flat,K^\fcc) .}
$$
Here the vertical maps are the multiplication maps on corresponding groups.

Consider the formal abelian schemes $\fB=\fA\times \fA$ and $\fB'=\fA'\times \fA'$. We do the same construction to get  their perfectoid universal covers $\CB^\perf$ and $\CB^{\perf\flat}$. 
The multiplication morphism $ \fB\to \fA$   induces $m:\CB^\perf\to \CA^\perf$. 
The multiplication morphism $ \fB'\to \fA'$   induces  $m':\CB^{\perf\flat}\to \CA^{\perf\flat}$.  By  \eqref{modpicond} and   \cite[Theorem 5.2]{Sch12}, $m'=m^\flat$.
By functoriality, we have a
 commutative diagram
 $$
    \xymatrix{
   \CB^\perf(K,K^\circ)  \ar[r]^{\rho_{  \CB^\perf}  } \ar[d]^{m}  & \CB^{\perf\flat} (K^\flat,K^\fcc)   \ar[d]^{m^\flat} \\
    \CA^\perf(K,K^\circ)\ar[r]^{ \rho_{  \CA^\perf} }    & \CA^{\perf\flat} (K^\flat,K^\fcc) .}
$$

We only need to show that this diagram can be  identified with the diagram we want.   For example we show that the top horizontal maps in the two diagrams coincide, i.e.   a commutative diagram
 $$
    \xymatrix{
   \CB^\perf(K,K^\circ)   \ar[r]^{  \rho_{  \CB^\perf}  } \ar[d]^{\cong}  &\CB^{\perf\flat} (K^\flat,K^\fcc) \ar[d]^{\cong} \\
 \CA^\perf(K,K^\circ)\times \CA^\perf(K,K^\circ)\ar[r]^{\rho \times  \rho }    & \CA^\perf(K,K^\circ)\times \CA^\perf(K,K^\circ) .}
$$
The projection  $ \fB=\fA\times \fA\to \fA$ to the $i$-th component, $i=1,2$, induces
$p_i:\CB^\perf \to \CA^\perf $. Easy to check that
 $$p_1\times p_2:\CB^\perf(K,K^\circ)\to \CA^\perf(K,K^\circ)\times\CA^\perf(K,K^\circ)$$ is a group isomorphism by passing to formal schemes.  Similarly we have an isomorphism  $$p_1'\times p_2':\CB^{\perf\flat} (K^\flat,K^\fcc)\to  \CA^\perf(K,K^\circ)\times \CA^\perf(K,K^\circ).$$
The commutativity  is   implied by that
$p_i'=p_i ^\flat$,  which is   from  \eqref{modpicond} and   \cite[Theorem 5.2]{Sch12}.      \end{proof}


 \subsection{Tilting and reduction}\label{Tilting and reduction}

Let   $k=\bar \BF_p$ and  let $W=W(k)$   be the ring of Witt vectors. 
   Let $A$ be an  abelian scheme over $W$,  $A_{K^\circ}$ be its base change to $K^\circ$,   $\CA$ be the adic generic fiber of $A_{K^\circ}$. 
   Let $A_k$ be the special fiber of $A$, and 
   $ A'$ be the base change $A_k \otimes K^{\flat\circ}$ with adic  generic fiber    $\CA' $. 
Since $$A_{K^\circ}\otimes (K^\circ/\varpi)\cong A\otimes_Wk\otimes_k(K^\circ/\varpi)\cong A'\otimes_{K^\fcc} (K^\fcc/\varpi^\flat),$$ 
we can apply the construction in  Lemma \ref{A.16} to the formal completions of $A_k\otimes_k{K^\circ}$ and $A'$.  
Then we have  the
perfectoid universal cover $\CA^\perf$  of  the $\varpi$-adic formal completion of $A_{K^\circ}$, the
perfectoid universal cover $\CA^{\perf\flat}$  of  the $\varpi^\flat$-adic formal completion of $A_{K^\fcc}$,
 and  the morphisms  $\pi_n:\CA^\perf\to \CA$,   $\pi_n':\CA^{\perf\flat}\to\CA'$ for each $n\in \BZ_{\geq0}$.
The following  well-known results  can be deduced from \cite{ST}.

 \begin{lem}\label{NOS}  (1) The inclusion  
 $A(W)\incl A ( {K^\circ})$ gives an isomorphism 
 $A(W)_{p'-\tor}\cong  A ( {K^\circ})_{p'-\tor}.$

 (2) The reduction map   gives   an isomorphism 
 $$\red: A(W)_{p'-\tor}\cong A( k)_{p'-\tor}.$$
 
(3) The natural inclusion  $A( k)\incl   A_{K^\fcc}(K^\fcc)$ gives an isomorphism  $A( k)_{p'-\tor} \cong A_{K^\fcc}(K^\fcc)_{p'-\tor} .$ 

 \end{lem}  
  
Now we relate reduction and tilting.
 \begin{lem}\label{comdiag} 
 Let the unindexed   maps in the following diagram be the naturals ones:
   $$
    \xymatrix{
    \CA ( K,K^\circ)_{p'-\tor}  & \ar[l]_{ \pi_n  } \CA^\perf (K, {K^\circ})_{p'-\tor}\ar[r]^{ \rho  } 
 & \CA^{\perf\flat} ( K^\flat, K^\fcc )_{p'-\tor}\ar[r]^{\pi_n'  }  &  \CA' ( K^\flat, K^\fcc )_{p'-\tor}  \\
  \ar[u]      A_{K^\circ}( {K^\circ})_{p'-\tor} & \ar[l]_{   }   A(W)_{p'-\tor}\ar[r]^{ \red  } 
 &A( k)_{p'-\tor}\ar[r]^{  }  & A_{K^\fcc}(K^\fcc)_{p'-\tor}\ar[u] 
.}
$$ 
Then  each map is a group isomorphism, and the diagram  is commutative (up to inverting the arrows).

  \end{lem}
\begin{proof}We may assume $n=0$.
 Definition \ref{adicgroup}, Lemma \ref{4isoms}, Proposition \ref{perfgroup} and Lemma \ref{NOS} give the isomorphisms. 
We only need to check the commutativity. 
  And we only need to check   the two maps from $A(W)_{p'-\tor}$ to $ \CA^{\perf\flat} ( K^\flat, K^\fcc )_{p'-\tor} $ are the same.
This follows from   Lemma \ref{tiltred2}. \end{proof}
 
 Similarly, we have the following   commutative diagram:
   \begin{equation} 
    \xymatrix{
   \CA^\perf (K, {K^\circ})     \ar[r]^{  \rho }  &  \CA^{\perf\flat} ( K^\flat, K^\fcc ) \ar[rd]^{  \cong}  &
    \\
   \ar@{^{(}->}[u] ^{\iota}   \vpl\limits_{[p]} A(W)  \ar[r]^{\red } \ar[d] ^{ \pi_0 } &
    \vpl\limits_{[p]}    \ar[d] ^{\pi'_n } A( k) \ar@{^{(}->}[r] & \ar[d] ^{\pi_n'} \vpl\limits_{[p]} \CA' ( K^\flat, K^\fcc )  \\
     \ar[d] ^{}   \bigcap\limits_{i=0}^{\infty}p^iA(W) &A( k)  \ar[d]^{[p^n]}\ar@{^{(}->}[r] &\ar[d]^{[p^n]} \CA' ( K^\flat, K^\fcc )\\
     A(W)  \ar[r]^{\red }&A( k) \ar@{^{(}->}[r] &   \CA' ( K^\flat, K^\fcc )
     .}
 \label{commallpts1} \end{equation} 
 Here  $\iota$ is induced from the inclusion $A(W)\incl \CA^\perf(K,K^\circ) $
and the isomorphism $\CA^\perf(K,K^\circ) \cong\vpl\limits_{[p]}   \CA(K,K^\circ)$ (see Lemma \ref{faca'}).
 Here and from now on   we regard  $\vpl\limits_{[p]} A(W)$ as a subset of  $\CA^\perf (K, {K^\circ}) $ via $\iota$, 
  $A(k) $ as a subset   $ \CA' ( K^\flat, K^\fcc )  $, and 
 $\vpl\limits_{[p]} A(k)  $ as a subset of $  \CA^{\perf \flat}( K^\flat, K^\fcc )$.  
  

      \section{Proof of    Theorem \ref{MM}}\label{proof}    
       In this section, we at first   prove  a lower bound   on   prime-to-$p$  torsion points in a subvariety. Then we  prove Theorem \ref{MM}.                 Let  $k=\bar \BF_p$,  $W=W(k)$ the ring of Witt vectors, and  $L=W[\frac{1}{p}]$. 

         \subsection{Results of Poonen, Raynaud and Scanlon }\label{estimates}

   \begin{thm}[Poonen \cite{Poo}]\label{BT}
  Let  $B$  be an abelian variety defined over $k$, and $V$  an irreducible  closed   subvariety of  $B$.  Let $S$ be a finite set of primes.  Suppose that $V$ generates $B$, 
 then the composition of $$V(k)\incl B(k)\xrightarrow{\bigoplus_{l\in S}\pr_{l}} \bigoplus_{l\in S}B[l^\infty]$$ is surjective, where $\pr_l$ is the projection to the $l$-primary component.
\end{thm}

      Let $A$ be   an abelian scheme   over $W$.   
Let $T=\bigcap\limits_{n=0}^{\infty}p^n(A(L)[p^\infty]),$   the maximal divisible subgroup of  $A(L)[p^\infty]$.  
Though  not needed, as an illustration, we note that by \cite[Exemples 5.2.3]{Ray83b}, $T=0$ if the $p$-rank of $A_k$ is 0 or if $A$ is a ``general ordinary abelian variety", and $T=A(L)[p^\infty]\cong  L/\BZ_p^{\dim A_L}$ if $A$ is the canonical lifting in Serre-Tate theory, see \ref {Classical Serre-Tate theory}.  
    \begin{lem} [{Raynaud \cite[Lemma 5.2.1]{Ray83b}}] \label{T''} %
  (1)  Let $T_o$ be the      subgroup of $A(\bar{L})[p^\infty]$ coming from the connected component of  the $p$-divisible group of $A $,
   then $T_o\bigcap T=0$.   
   
   (2) As a subgroup of $A(  \bar L)[p^\infty]$,
   $T$ is  a $\Gal(\bar{L}/L)$-direct summand.
   
      \end{lem} 
 Note that  \begin{equation}\bigcap\limits_{n=0}^{\infty}p^n(A(W)_\tor)= A(W)_{p'-\tor}\bigoplus \bigcap\limits_{n=0}^{\infty}p^n(A(W)[p^\infty]).\label{shabi}\end{equation}

     \begin{cor} \label{T'''} 
        The following reduction map is injective $$\red:   \bigcap\limits_{n=0}^{\infty}p^n(A(W)_\tor)   \to   A(k).$$
\end{cor}
   
 
  Let   $Z \subset A_L$ be a closed subvariety. 
   \begin{lem} [{Raynaud \cite[8.2]{Ray83}}]\label{finalestimate} Let  $T'$ be  a $\Gal(\bar{L}/L)$-direct summand  such that as $\Gal(\bar{L}/L)$-modules
   $$ A(\bar{L})_{ \tor}= A(\bar{L})_{p'-\tor}\bigoplus T\bigoplus T' .$$ 
    If $Z$ does not contain  any translate of   a nontrivial abelian subvariety of $A_L$,
   there exists a positive integer $N$ such that 
   the order of the $T'$-component of every element in $Z(\bar{L})_\tor$ divides $p^N$.
    \end{lem}     
    \begin{rmk}Lemma  \ref{finalestimate} is used by Raynaud \cite{Ray83} to reduce the Manin-Mumford conjecture to a theorem (see  \cite[Theorem 3.5.1]{Ray83}) obtained 
   by studying $p$-adic rigid analytic properties  of universal vector extension
      of an abelian variety.
    \end{rmk}
 Let  $K$ and $K^\flat$  be the  perfectoid fields in \ref{perftheory}.
 Let  $\CA$ be the adic generic fiber of $A_{K^\circ}$. Let  $Z^\Zar$ be the Zariski closure of 
 $Z$ in $A$, and  $\CZ$ the adic generic fiber of   $Z^\Zar_{K^\circ}$,   .
   For $\ep\in K^\times, $ let $\CZ_\ep$ be  the $\ep$-neighborhood of $Z_K$ in $\CA$ as in  Definition \ref {globalnbhd}.
 By Lemma \ref{distep}, a result  of Scanlon   \cite{Sca0} on the Tate-Voloch conjecture implies the following lemma.
   \begin{lem}[Scanlon   \cite{Sca0}] \label{Scalem}There exists  $\ep\in K^\times, $  such that  $\CA(K,K^\circ)_{p'-\tor}\bigcap \CZ_\ep\subset \CZ.$

   \end{lem}
  \begin{rmk} 
  
  The proofs of Poonen's  result   and Scanlon 's  result  are   independent of Theorem \ref{MM}.

           \end{rmk} 

  \subsection{A   lower bound  }\label{the proof of  Proposition}
           
Define \begin{equation}  \Lambda  :=Z^\Zar(W)\bigcap   \bigcap\limits_{n=0}^{\infty}p^n(A(W) _\tor),\ 
\Lambda_\infty :=\iota\left( \pi_0^{-1}(\Lambda)\right) ,\label{deflambda}\end{equation}    
where $\pi_0$ and $\iota$ are  as in  the left column of diagram \eqref{commallpts1}.
Then $\rho( \Lambda_\infty)$ is contained in (the image of) $  \vpl\limits_{[p]}  A( k) $ by diagram \eqref{commallpts1}. Now let $ \Lambda_n :=\pi_n'(\rho( \Lambda_\infty)) $. Then $\Lambda_n$ is  contained in (the image of)  $A( k)$. 
   Let $\Lambda^\Zar_n$ be  the Zariski closure of $\Lambda _n$ 
   in $A_{k}$.

        \begin{prop} \label{techprop} There exists a positive integer $n$ such that   
$$\pi_0\left(\rho^{-1}\left({\pi_n'}^{-1}\left(\Lambda _n^\Zar (k) _{p'-\tor} \right)\right )\right)\bigcap     \CA ( K,K^\circ)_{p'-\tor} \subset \CZ .$$ 
  \end{prop} 
     \begin{proof}  
     Let $\CU$ be a  finite affine open  cover of $A$  by  affine open subschemes  flat  over $W$.     Let $U\in \CU$. The restriction of $\CA^\perf$ over the adic generic fiber of $U_{K^\circ} $ is a perfectoid  space $\CX=\Spa(R,R^+)$ whose tilt  satisfies Assumption \ref{asmp4} (see Lemma \ref{A.16} and the discussion above it).
     Let $\CI$ be the ideal sheaf of  $Z^\Zar$.
     Let $f\in \CI(U)$. Regard $f$ as in $R$. 
     By  definition of $\Lambda _n$, we can apply 
         Proposition \ref{techprop2}  to  $f$ and $\Lambda _n$. 
         Varying $U$ in $\CU$ and varying  $f$ in a finite set of   generators of $\CI(U)$,
        Proposition \ref{techprop2} implies that  for every $\ep\in K^\times, $  there exists a positive integer $n$ such that   
$$\pi_0\left(\rho^{-1}\left({\pi_n'}^{-1}\left(\Lambda _n^\Zar (k)  \right)\right )\right)\subset \CZ_\ep .$$ 
Then   Proposition \ref{techprop}  follows from   Lemma \ref{Scalem}.
                \end{proof}

         Our lower bound  on the size of the set of prime-to-$p$ torsions in $Z$ is as follows.

     \begin{prop}\label{mainprop} Let  $p>2$.   Assume that $Z$ contains the unit $0\in A_L$.
 
                        (1) Assume  $\Lambda$ is infinite.  For every   prime number $l\neq p$, the image of the composition of
    \begin{equation}   Z^\Zar(W) _{p'-\tor}\incl A(W)_{p'-\tor}\xrightarrow{\pr_l} A(W)[l^\infty]
    \label{intro3}
 \end{equation} 
contains a translate of a free $\BQ_l/\BZ_l$-submodule of $A(W)[l^\infty]
$ of rank at least  $2$.
Here the  map $\pr_l$ is the projection to the $l$-primary component.
   
   (2)        Assume that the image of the composition of     \begin{equation*}  \Lambda\incl A(W)_{\tor}\xrightarrow{\pr_p} A(W)[p^\infty]  \end{equation*}  
   contains a translate of a free $ L/\BZ_p$-submodule of rank $r$.  For every   prime number $l\neq p$, the image of the composition  of \eqref {intro3}
          contains a translate of  a free $\BQ_l/\BZ_l$-submodule of $A(W)[l^\infty]
$ of rank  $2r$.

  \end{prop} 
 
  \begin{proof}
 
  Fix a large $n$   such that  \begin{equation} \pi_0\left(\rho^{-1}\left({\pi_n'}^{-1}\left( \Lambda _n^\Zar (k)_{p'-\tor}  \right)\right )\right)\bigcap   \CA ( K,K^\circ)_{p'-\tor} \subset \CZ ( K,K^\circ)_{p'-\tor} 
\label{chang}
\end{equation} as in Proposition \ref{techprop}.   
   Let $ X  $ be  the image of the left hand side of \eqref{chang} via the
composition of \begin{equation*}    \CA ( K,K^\circ)_{p'-\tor} 
 \cong A(W)_{p'-\tor}  \xrightarrow{\pr_l}   A(W)[l^\infty].
 \end{equation*}     Then $X $ is contained in the image of the composition of \eqref{intro3}.

 To prove (1), we only need to prove the following claim:   $X$ contains a translate of a free $\BQ_l/\BZ_l$-submodule of $A(W)[l^\infty]
$ of rank at least  $2$ for every $l$.  

  By   diagram \eqref{commallpts1}, 
we have $\Lambda _0= \red ( \Lambda) .$
Since $p>2$,   by Corollary \ref{T'''},   $\Lambda_0$ is infinite. Since $\Lambda_0=[p]^n(\Lambda_n)$, 
$\Lambda_n$ is infinite. There exists $a\in A(k)$ such that   an irreducible  component of  $  \Lambda _n^\Zar   +a$ 
(is contained and)
generates a nontrivial abelian subvariety  $A'$ of $A_k$. Since $Z$   contains the unit $0\in A_L$, $  \Lambda _n^\Zar $ contains  the unit $0\in A_k$ and  $A'$ contains $a$.
Let $a_p$ be the $p$-primary part of $a$ and $a_{p'}=a-a_p$.
By  Theorem \ref{BT} (for $ \Lambda _n^\Zar   +a\subset A'$ and   $S=\{p,l\}$), 
 the image of \begin{equation}  \Lambda _n^\Zar (k) +a\xrightarrow{\pr_l\bigoplus \pr_p}A (k)[l^\infty]\bigoplus A (k)[p^\infty]\label{zhejuhua}\end{equation}
  contains $M\bigoplus \{a_p\}$, where 
  $M$ is 
    a free $\BQ_l/\BZ_l$-submodule of $A(k)[l^\infty]$ of rank at least $2$.  
    Thus     $$\left(\pr_l\bigoplus \pr_p\right)\left ( \Lambda _n^\Zar (k) +a_{p'}\right)\supset  M\bigoplus\{ 0\}.$$
    We claim:   $$ \pr_l\left (\left (  \Lambda _n^\Zar (k) +a _{p'}\right)_{p'-\tor}\right)\supset  M .$$ 
Indeed, write $b\in \Lambda _n^\Zar (k) +a_{p'}$ as the sum $b_p+b_{p'}$ of $p$-primary part  and prime-to-$p$ part. Then 
$\left(\pr_l\bigoplus \pr_p\right) b=\pr_l(b_{p'})+b_p$. If this is $x+0\in M\bigoplus\{ 0\}$, then $b_p=0$, and $b=b_{p'}$.
Thus $ \pr_l(b)=x\in M$. The claim is proved.
 By the claim,  $$ M-\pr_l (a_{p'})\subset Y  :=  \pr_l\left( \Lambda _n^\Zar (k)  _{p'-\tor}\right)  . $$  
   By   Lemma \ref{comdiag}, 
$X  $
 contains  the preimage of $[p]^n(Y )$ under   the isomorphism $\red: A(W)_{p'-\tor}\cong A( k)_{p'-\tor}.$
    Thus we proved the claim above. 

   To prove (2), we only need to prove the following claim:  $X $ contains a translate of a free $\BQ_l/\BZ_l$-submodule of $A(W)[l^\infty]
$ of rank at least  $2r$ for every $l$. 

By   diagram \eqref{commallpts1}, 
we have $\Lambda _0= \red ( \Lambda) .$  Since $p>2$,   by Corollary \ref{T'''}  
and the assumption on $\Lambda$,
 $\pr_p(\Lambda_0)$ contains a translate of a free $ L/\BZ_p$-submodule of rank $r$. 
 Let $V_1,...,V_m$ be the irreducible components of 
  $\Lambda_0^\Zar$. 
Let  $A_i$ be the  minimal abelian subvariety of $A_k$ such that a certain translate of $A_i$ contains $V_i$. 
Since 
 the $p$-rank of  $A_i$ is at most  its dimension,      at least one $A_i$ is 
  of dimension at least $r$.    
Since $\Lambda_0=[p]^n(\Lambda_n)$, 
 there exists $a\in A(k)$ such that   an irreducible  component of  $  \Lambda _n^\Zar   +a$ generates an abelian subvariety of $A_k$  of dimension at least $r$.   
Then we prove (2) by copying the proof of (1) above, starting from  the sentence containing \eqref{zhejuhua}.
The only modification needed is that the rank of $M$ should be at least $2r$.
      \end{proof} 
    
         \subsection{The proof of Theorem \ref{MM}}\label{proofcurve}

   Now we prove Theorem \ref{MM}. By the argument in \cite{PZ}, we only need to prove the following weaker theorem.
   We save the  symbol $A $ for the proof.
  \begin{thm}\label{curveMM1} Let $F$ be number field. Let 
   $B$ be an abelian variety over  $F$  and $V $ a closed subvariety    of  $B$.   
 If  
  $V$  does not contain any translate of an  abelian subvariety of $B$ of positive dimension, then $V$ contains only finitely many torsion points of $B$.
  \end{thm}

 \begin{proof}

    We only need to prove the theorem  up to replacing $V$ by  a multiple.
    
  Let $v$ be a place of  $F$ unramified over a prime number $p>2$ 
  such that  $B$ has good reduction.      Let $A$ be   the base change to $W $ of the  integral smooth model of $B$ over $\CO_{F_v}$. 
  Let $Z=V_L\subset A$.
  By \eqref{shabi} and Lemma  \ref{finalestimate},   up to replacing $V$ by   $[p^N]V $ for $N$ large enough, we may assume that $Z^\Zar(W)_\tor\subset  \bigcap\limits_{n=0}^{\infty}p^n(A(W)_{\tor}).$ Thus $\Lambda=Z^\Zar(W)_\tor$, where $\Lambda$ is defined as in \eqref{deflambda}. Suppose that $V$ contains infinitely many torsion points. Then $\Lambda$  is infinite. Up to replacing $V$ by   $[p^N]V $, we may assume that $Z$ contains the unit $0\in A_L$.
    Now we want to find a contradiction.
  By Proposition \ref{mainprop} (1), 
  for every prime number $l\neq p,$     the composition of \begin{equation*} Z(L)_{p'-\tor}\incl A(L)_{p'-\tor}\xrightarrow{\pr_l} A(L)[l^\infty]
 \end{equation*} 
contains a translate of a free $\BQ_l/\BZ_l$-submodule of $A(L)[l^\infty]
$ of rank  $2$.   

Let $u$ be another place of  $F$, unramified over an odd prime number  $  l\neq p$,
  such that  $B$ has good reduction at $u$.  Let $B_o$ be the reduction. Let $M$ be the completion of the maximal unramified extension   of $F_u$ and $\bar M$ its algebraic closure.  
  Then  the composition \begin{equation*} V(\bar M)_\tor \incl B(\bar M)_\tor \xrightarrow{\pr_l} B(\bar M)[l^\infty]
 \end{equation*} 
   contains a translate $G$ of a free $\BQ_l/\BZ_l$-submodule   of $B(\bar M)[l^\infty]
$ of rank  $2$.   
    Let   $T=\bigcap\limits_{n=0}^{\infty}l^n(B(M)[l^\infty]).$   
 By  Lemma  \ref{finalestimate} (applied  to $l$, $M$ instead of $p$,  $L$), up to replacing $V$ by   $[l^N]V $ for $N$ large enough,
  $G  $ is contained in $T$.  
By \eqref{shabi} (applied  to $l$, $M^\circ$ instead of $p$,  $W$),   the image of the composition of
  \begin{equation*} V( M)  \bigcap  \bigcap\limits_{n=0}^{\infty}l^n(B(M)_\tor) \incl B(M) \xrightarrow{\pr_l} B(M)[l^\infty]
 \end{equation*}  contains $G$.
 By Proposition \ref{mainprop} (2)  (applied  to $l$, $M^\circ$ instead of $p$,  $W$),  
  for every prime number $q\neq l$,  the composition \begin{equation*} V(  M)_{l'-\tor}\incl B(  M)_{l'-\tor}\xrightarrow{\pr_q} B(  M)[q^\infty]
 \end{equation*} 
   contains  a translate of a free $\BQ_q/\BZ_q$-submodule   of rank  $4$.    Repeating this process (use more places  or only work at $v$ and $u$), we get a contradiction as $A$ is of finite dimension.
    \end{proof}

   \section{Ordinary perfectoid Siegel space and Serre-Tate theory}
\label{The ordinary perfectoid Siegel space and Serre-Tate theory}
   Let $\BA_f$ be the ring of finite adeles of $\BQ$, $U^p\subset \GSp_{2g}(\BA_f^p)$ an open compact subgroup contained in the congruence subgroup of level-$N$ for some $N\geq 3$ prime to $p$.
 Let $X=X_{g,U^p} $  over $\BZ_p$ be 
 the 
 Siegel moduli space of principally polarized $g$-dimensional abelian
varieties over $\BZ_p$-schemes with level-$U ^p$ -structure. Let   $X_o$ be special fiber of $X$.  
We will use  the  perfectoid fields defined in \ref{perftheory}.    We briefly recall some notations.  Let  $k=\bar \BF_p$, $W=W(k)$  the ring of Witt vectors, $L$ the  fraction field of $W$, and       $L^\cycl$ the field extension of $L$ by adjoining all $p$-power-th roots of unity.   Let $K$ be
  the $p$-adic completion of $L^\cycl$ which is a perfectoid field. Then $K^\flat= k((t^{1/p^\infty})) $ is the tilt of $K$.
Fix a primitive $p^n$-th root of unity  $\mu_{p^n}$  for every positive integer $n$
    such  that $\mu_{p^{n+1}}^p=\mu_{p^n}$.
  
      \subsection{Ordinary perfectoid Siegel space}
\label{Ordinary perfectoid Siegel space}
      Let $X_o(0) \subset X_o$ be the ordinary locus.      Let $\fX(0)$ over $\BZ_p$ be   
   the open  formal subscheme of the formal completion of $X$  along  $X_o$ defined by the condition that every  local lifting of the Hasse invariant is invertible (see \cite[Definition 3.2.12, Lemma 3.2.13]{Sch13}).    
       Then    $\fX(0)/p= X_o(0)$ (see  \cite[Lemma 3.2.5]{Sch13}). 
Let $ \widehat{X_o(0)_{K^\fcc}}$ be  the $\varpi^\flat$-adic formal completion of $X_o(0)_{K^\fcc}$.
Let  $\CX(0) $ and $\CX'(0) $ be the adic genric fibers of $\fX(0)_{K^\circ}$ 
and $\widehat{X_o(0)_{K^\fcc}}$  respectively.  
 

Let  $\Fr:X_o(0)\to   X_o(0)$ be   the (relative)
   Frobenius morphism (note that $X_o(0)$ is defined over $\BF_p$). %
   Let                 $  \Fr^\can:\fX(0)\to\fX(0)$ be given by the functor sending an abelian scheme $A$ to its quotient by the connected subgroup scheme of $A[p]$.
              Then $  \Fr^\can/p=\Fr$. We also use $  \Fr^\can$ and $\Fr$ to denote their base changes to $K^\circ$ and $K^\fcc$ respectively.
  Let    $$ \tilde \fX(0):=\vpl_{\Fr^\can}\fX(0)_{K^\circ},\ \tilde \fX'(0):=\vpl_{\Fr} \widehat{X_o(0)_{K^\fcc}},$$ 
  where  the inverse  limits are taken   in the categories of   $\varpi$-adic  and $\varpi^\flat$-adic formal schemes respectively.  
  Here 
  $\varpi=\mu_p-1$    and $\varpi^\flat=t^{1/p}$.  
    By \cite[Corollary 3.2.19]{Sch13},  the  corresponding  adic genric fibers  $\CX(0)^\perf$ and $\CX'(0)^\perf$ of  $ \tilde \fX(0)$ and $ \tilde \fX'(0)$   are perfectoid spaces. Moreover, $\CX'(0)^\perf=\CX(0)^{\perf,\flat}$,  the tilt of $\CX(0)^\perf$.   
    Then we have   the natural projections
\begin{equation} 
 \pi:\CX(0)^\perf\to \CX(0),\ \pi':\CX(0)^{\perf,\flat}\to \CX'(0)\label{natpi}.\end{equation}    
  We also have  a natural map between the underlying sets defined in \ref{srho}\begin{equation}  \rho_{\CX(0)^\perf}:|\CX(0)^\perf|\to|\CX(0)^{\perf,\flat}|.\label{natrho}.\end{equation}    
  (The map  $\rho_{\CX(0)^\perf}$ is  in fact a homeomorphism and we do not need this fact.)

   \subsection{Classical Serre-Tate theory} \label{Classical Serre-Tate theory}
   
  We   use the adjective ``classical"   to indicate the Serre-Tate theory \cite{Kat} discussed in this subsection,  compared with Chai's global Serre-Tate theory to be discussed  in \ref{GST}. 
  
   Let $R$ be an Artinian local ring with maximal ideal $\fm$ and residue field $k$.
   Let $ A/\Spec R$ be an abelian scheme with ordinary special fiber $A_k$.  Let $A_k^\vee$ be the dual abelian variety of $A_k$.
  There is a   $\BZ_p$-module
 morphism from the product of Tate-modules $  T_p A_{k}\otimes T_p A_k^\vee$   to $1+\fm$   
constructed by Katz \cite{Kat}.  We call this morphism the  classical  Serre-Tate  coordinate system   for $A/\Spec R$.
 If
$ A/\Spec R$ is moreover a principally polarized   abelian scheme,
 the    Serre-Tate  coordinate system   for $A/\Spec R$ is a    $\BZ_p$-module
morphism 
\begin{equation}  q_{A/\Spec R} :  \Sym^2(T_p A_{k}) \to 1+\fm.\label{60}\end{equation} 
  
Let $x\in X_o(0)(k)$,  and let $A_x $ be the  corresponding  principally polarized   abelian variety.   
Let $\fM_x$  be the formal completion of $X$ at $x$,  and $\fA/\fM_x$ the formal universal  deformation of $A_x$. 
Then  as part of the construction of $q_{A/\Spec R}$, there is an isomorphism of formal schemes over $W$:
\begin{equation}\fM_x\cong \Hom_{\BZ_p}( \Sym^2(T_p A_{x}),\hat\BG_m)\label{stiso},\end{equation}
where $\hat\BG_m$ is the formal completion of the multiplicative group scheme over $W$  along the unit section.
In particular, 
$\fM_x$ has  a formal torus structure.  
Moreover, if  $A_k\cong A_x$ in \eqref{60},  then  \eqref{60}  is the value of \eqref{stiso} at the morphism $\Spec R\to \fM_x$ induced by $A$. Let  $\CO(\fM_x)$ be the coordinate ring of $\fM_x$, and let   $\fm_x$ be the maximal ideal of $\CO(\fM_x)$. 
From \eqref{stiso}, we have a morphism of $\BZ_p$-modules: $$ q=q_{ \fA/\fM_x}:\Sym^2(T_p A_{x}) \to 1+\fm_x.$$
Fix a basis $\xi_1,...,\xi_{g(g+1)/2}$ of $\Sym^2(T_p A_{x})$.
\begin{prop}[{\cite[3.2]{dJN}}]\label{dJN3.2}  Let $F$ be a finite extension of $L$ with ring of integers $F^\circ$.  
Let   $y^\circ\in X(F^\circ)$ with generic fiber $y$. 
Suppose that     $y^\circ\in \fM_x(F^\circ)$.
Then $y$ is a CM point if and only if   $q(\xi_i)(y^\circ)$ is a $p$-primary root of unity for $i=1,...,g(g+1)/2$.
\end{prop}
Thus every ordinary CM point    is contained in $X(L^\cycl)$.   For  an ordinary CM point  $y\in X(L^\cycl)$.
there is a unique $y^\circ\in X (K^\circ)$ whose generic fiber is $y_K\in X(K)$.  We   regard $y^\circ$ as a point in $\fX(0)(K^\circ)$  and $y_K$ as a point in $\CX(K,K^\circ)$ via Lemma \ref{algpoints}.

 \begin{defn}\label{generator} 

 Let  $a =(a ^{(1)},...,a^{(g(g+1)/2)})\in\BZ_{\geq 0}^{g(g+1)/2}$.
 
 (1)   An ordinary CM point $y\in X(L^\cycl)$ with reduction $x$ 
   is called of order  $p^a$   w.r.t. the basis  $\xi_1,...,\xi_{g(g+1)/2}$ 
        if 
   $q(\xi_i)(y^\circ) $  is a primitive $p^{a^{(i)}}$-th root of unity  for each $i=1,...,g(g+1)/2$.
  If moreover   $q(\xi_i)(y^\circ)=\mu_{p^{a^{(i)}}}$, $y$  
   is  called a $\mu$-generator w.r.t. the basis  $\xi_1,...,\xi_{g(g+1)/2}$.
   
   (2)  Assume that $a$ is non-increasing  so that $q(\xi_{i+1})(y^\circ)$ is an $r^{(i)}$-th power of $q(\xi_{i})(y^\circ)$ for 
   some (non-unique)  $r^{(i)}\in \BZ_p$, $i=1,. . . ,g(g+1)/2-1$. We call $(r^{(1)},. . . ,r^{( {g(g+1)/2-1})})\in \BZ_p^{g(g+1)/2-1}$ a ratio of $y$ w.r.t.  the basis  $\xi_1,...,\xi_{g(g+1)/2}$.
   \end{defn}
   It is clear that if $a$ is  non-increasing,  then the usual $p$-adic absolute value  $|r^{(i)}|_p=p^{a ^{(i+1)}-a ^{(i)}}$.

Let $T_i =q(\xi_i)-1\in \fm_x.$  Then we have an isomorphism \begin{equation}\CO(\fM_x)\cong W[[T_1,...,T_{g(g+1)/2}]]\label{6.1}.\end{equation}
Let $ \widehat {X_o(0)}_{/x}$ be  the  formal completion  of $X_o(0)$ at $x$. 
Restricted to  $ \widehat {X_o(0)}_{/x}$, \eqref{6.1} gives
 an isomorphism \begin{equation}\CO(\widehat {X_o(0)}_{/x})\cong k[[T_1,...,T_{g(g+1)/2}]].\label{6.3}\end{equation}
Let   $U\to X_o(0)$ be an \etale morphism, $z\in U(k)$ with image $x$.
 Then 
\eqref{6.3} gives an isomorphism \begin{equation}\CO(\widehat {U}_{/z})\cong k[[T_1,...,T_{g(g+1)/2}]]. \label{63}\end{equation}
Let $A_z$ be the pullback of $A_x$ at $z$. Then we naturally have $T_p A_z\cong T_p A_x$. Thus we also regard $\xi_1,...,\xi_{g(g+1)/2}$ as a basis of  $\Sym^2(T_p A_{z})$.

\begin{defn}\label{realization'}  
We call \eqref{63} the  realization of the classical Serre-Tate  coordinate system of 
 $ \widehat {U}_{/z}$ at the basis $\xi_1,...,\xi_{g(g+1)/2}$ of  $\Sym^2(T_p A_{z})$.
\end{defn}     

We have another description of \eqref{63}. Let $I_n$ be a descending  sequence of open ideals of $\CO(\widehat {U}_{/z})$ defining the topology of $\CO(\widehat {U}_{/z})$. Let $R_n:=\CO(\widehat {U}_{/z})/I_n$, 
let $A_n $  be the pullback of the formal universal principally polarized abelian scheme over $\fM_x$ to $\Spec R_n$
with special fiber $A_z$.  Let $$q_{A _n/\Spec R_n}: \Sym^2(T_p A^\univ_x)\to   R_n^\times $$ be the  classical Serre-Tate coordinate system  of $A _n/R_n$. 
    Then $ q_{A_n/\Spec R_n}(\xi_i)-1=T_i(\mod I_n)$.
   Thus the sequence   $\{ q_{A_n/\Spec R_n}(\xi_i)-1\}_n $   gives an element in $ \CO(\widehat {U}_{/z})\cong \vpl_n  R_n,$
  which equals $T_i$. 

 \subsection{Tilts of   ordinary  CM points}\label{Tilting canonical liftings}
 Let $ \widehat {X_o(0)_{K^{\fcc}}}_{/x}$ be  the  formal completion  of ${X_o(0)_{K^{\fcc}}}$ at $x$.
By \eqref{6.3}, we  have
\begin{equation}\CO\left (\widehat {X_o(0)_{K^{\fcc}}}_{/x}\right)\cong K^{\fcc}[[T_1,...,T_{g(g+1)/2}]].\label{cDx}\end{equation}
 Let $\cD_x$ be the adic generic fiber of $\widehat {X_o(0)_{K^{\fcc}}}_{/x}$. 
Then $\cD_x$ is  is an adic subspace of $\CX'(0)$ in the sense of Definition \ref{adicdef}.
  Moreover,  \eqref{cDx}  and Lemma \ref {algpoints} imply an isomorphism \begin{equation}\cD_x(K^\flat, K^\fcc)\cong K^{\flat\dbc,g(g+1)/2}. \label{cdreal}\end{equation}

\begin{lem}\label{6.2.3} 
Let $y\in X(L ^\cycl)$ be an ordinary CM point with reduction $x$. 

(1) For every $\tilde y\in \pi^{-1}(y_K)\subset \CX(0)^\perf$,  we have $$\pi'\circ \rho_{\CX(0)^\perf}(\tilde y)\in \cD_x .$$

(2)  Let  $a =(a ^{(1)},...,a^{(g(g+1)/2)})\in\BZ_{\geq 0}^{g(g+1)/2}$ and  $I\subset \{1,2,...,g(g+1)/2\}$   the subset of $i$'s such that   $a^{(i)}=0$.  Let $y$ be  a $\mu$-generator    of order $p^a$   w.r.t.  the basis $\xi_1,...,\xi_{g(g+1)/2}$ (see Definition \ref{generator}). There exists $\tilde y\in \pi^{-1}(y_K)$ such that  via the isomorphism   \eqref{cdreal},
 the $i$-th coordinate of  $\pi'\circ \rho_{\CX(0)^\perf}(\tilde y) $ is 0 for $i\in I$ and is $ t^{1/p^{a^{(i)}}}$ for $i\not\in I$.
\end{lem}
\begin{proof}

We recall the effect of  $\Fr^\can$ on $\fM_x$   (see \cite[4.1]{Kat}).
 Denote  $\fM_{x}$  by $\fM_{A_x}$.
   Let $\sigma\in \Aut(k)$ be the Frobenius.  Let $A_x^{(\sigma)}:= A_x \otimes_{k,\sigma} k$ be the base change by   $\sigma$.
Then  $\Fr^\can $,  restricted   to  $\fM_{A_x}$ gives a morphism $\Fr^\can: \fM_{A_x}\to \fM_{A_x^{(\sigma)}}$ over $W$ \cite[p 171]{Kat}.
  Let    $\sigma(\xi_1),...,\sigma(\xi_{g(g+1)/2})$ be the induced basis
   of $\Sym^2(T_p A_{x}^{(\sigma)}[p^\infty])$. 
 Then
 \cite[Lemma 4.1.2]{Kat} implies that \begin{equation} \Fr^{\can,*}(q (\sigma(\xi_i)))=q (\xi_i)^p.\label{Qp}\end{equation}

 We associate a perfectoid space to $\fM_x$.
 Let 
\begin{equation*}\tilde\fM_x:=\vpl\limits_{\Fr^\can} \fM_{ A_x^{(\sigma^{-n})}}  .\end{equation*} 
 By  a similar (and easier) proof as the one for \cite[Corollary 3.2.19]{Sch13},  the adic generic fiber $\CM_x^\perf$ of $\tilde\fM_{x,K^\circ}$
is a perfectoid space. 
Moreover, let $\CM_x'^\perf$ be the adic generic fiber  of $\vpl\limits_{\Fr } \widehat {X_o(0)_{K^{\fcc}}}_{/x}$.
  Then   $\CM_x'^\perf$ is the tilt of $ \CM_x^\perf$.  
    By Lemma    \ref{tiltmor}, the tilting process commutes with restriction to an open subspace.  Thus to prove  Lemma \ref{6.2.3}, we we only need to consider the tilting between  $ \CM_x^\perf$ and $\CM_x'^\perf$.  
 Then Lemma \ref{6.2.3} follows from the cases $c=0$ and $c=1$  of Lemma \ref{2316} (which deals with closed units discs while here we are dealing with open unit discs so that we apply    \ref{tiltmor} again).
\end{proof}

\subsection{Global  Serre-Tate  theory}\label{GST}
\subsubsection{The algebraic and geometric formulations}
 Now we review Chai's globalization of Serre-Tate  coordinate system in characteristic $p$ \cite{Cha}. 
Let $U$ be a  $\BF_p$-scheme.
Let $A/U$ be  an   abelian scheme whose relative dimensions on   connected components of $U$ are the same.  
Define  
$$\nu_U=\vpl _{n}\Coker([p^n]:\BG_m\to \BG_m),$$
which is  a $\BZ_p$-sheaf    on $U_\et$.

\begin{eg}\label{GSTeg} 
(1)     Let $m\geq n$ be positive integers, and   $U_0=\Spec k[T]/T^{p^n}$. Then the $p^m$-th power of  an element in $(k[T]/T^{p^n})^\times$ with constant term $b$ is $b^{p^m} $. 
Thus  
    \begin {equation}\nu_{U_0}(U_0)=(k[T]/T^{p^n})^\times/k^\times\cong 1+T(k[T]/T^{p^n}).\label{phix}\end{equation}

 (2) Let $B$ be  an    $\BF_p$-algebra, $U=\Spec B $ and  
 $U'=\Spec B[T]/T^{p^n}$.  
For $m\geq n$, consider the map $$B^\times/(B^\times)^{p^m}\bigoplus
(1+TB[T]/T^{p^n})\to (B[T]/T^{p^n})^\times/((B[T]/T^{p^n})^\times)^{p^m}$$
defined  by
 $(a,f)\mapsto af.$
Easy to check that this  is a group isomorphism.
In particular,  \begin {equation}\nu_{U'}(U')\cong\nu_{U}(U)\bigoplus
(1+TB[T]/T^{p^n}) .\label{phixB}\end{equation}
 
  (3) For every       $z\in U(k) $,   $\{z\}\times _{U}U' \cong U_0$. 
 Then the restriction of the isomorphism \eqref{phixB}   at $z$ is the isomorphism     \eqref{phix}.
 \end{eg}

Suppose  $A/U$ is ordinary. 
Let  $T_p A[p^\infty]^\et$ be  the Tate module attached to the maximal \etale quotient of the $p$-divisible group  $A[p^\infty]$. 
The     global  Serre-Tate  coordinate system for $A/U$ is a homomorphism of $\BZ_p$-sheaves 
$$ q_{A/U}:T_p A[p^\infty]^\et\otimes   T_p A^\vee[p^\infty]^\et\to \nu_U$$  
   constructed by Chai \cite[2.5]{Cha}. 
   Let   $U_0=\Spec k[T]/T^{p^n}$.  
Let $A/U_0$ be an ordinary abelian scheme, and $A_k$  the special fiber of $A$. 
Then 
\begin{equation} T_p A[p^\infty]^\et\otimes T_p A^\vee[p^\infty]^\et \cong T_p A_k[p^\infty]\otimes T_p A_k^\vee[p^\infty],\label{612}\end{equation}
 where the right hand side  is regarded as  a constant sheaf.

\begin{lem}[{\cite[(2.5.1)]{Cha}}]\label{GSTeg'}       The  morphism of $\BZ_p$-modules $$T_p A_k[p^\infty]\otimes T_p A_k^\vee[p^\infty]\to \nu_U(U)\cong  1+T(k[T]/T^{p^n})$$  induced from $q_{A/U_0} $ via \eqref{612}   coincides with  the classical Serre-Tate coordinate system (see \eqref{60}).

 \end{lem}

 The geometric formulation of global  Serre-Tate  coordinate system is as follows.  
 Let $A^{\univ}$ be the universal principally polarized abelian scheme over $X_o(0)$,  and $\hat A^{\univ}$   the formal completion of $A^{\univ}$ along the zero section which is a formal torus over $X_o(0)$.
  Then the  sheaf   of polarization-preserving  $\BZ_p$-homomorphisms
 between $T_p A^{\univ}[p^\infty]^\et$ and  $\hat A^{\univ}$  is a formal torus over $ X_o(0)$ of dimension $\frac{g(g+1)}{2}$. Let us call it $\fT_1$. Let $\Delta$ be the diagonal embedding of $X_o(0)$
into $X_o(0)\times X_o(0)$,  and let 
$\fT_2$ be the formal completion of $X_o(0)\times X_o(0)$  along this embedding.    
     \begin{prop}[{\cite[Proposition 5.4]{Cha}}] \label{geogl}
  There is a canonical isomorphism   $\fT_1\cong \fT_2.$
In particular,  $\fT_2$  has a formal torus structure over the first $X_o(0) $. 

\end{prop} 
\subsubsection{Igusa tower}
In order to have sections of  the \etale $\BZ_p$-sheaf   $T_p A[p^\infty]^\et$ over $U$,
or equivalently to trivialize the formal torus, we need to pass to  the Igusa tower, defined as follow.  
For $n=0,1,...,\infty$, let $\fI_n$  be the functor   assigning to every $k$-algebra $R$  the set of isomorphism classes of pairs 
$$\{(A,\vep):A\in X_o(0)(R),\ \vep:A[p^n]\cong \hat\BG_{m,R}^g[p^n]\}.$$
By \cite[8.1.1]{Hid},  for $n<\infty$ (resp. $n=\infty$)
 the functor $\fI_n$ is represented by a   $k$-scheme (which we still denote by $\fI_n$)  finite (resp. profinite)   Galois over 
$X_o(0)$ with Galois group $\GL_g(\BZ/p^n\BZ)$ (resp.  $\GL_g(\BZ_p)$). 
And $\fI_n$ is  known as the Igusa scheme of level $n$.

  \subsubsection{Realization of  the global  Serre-Tate  coordinate system at a basis}\label{GSTapp} 
Let $U_0 $ be  an affine open subscheme of  $X_o(0)$.  Let   $U=\Spec B:=\fI_\infty|_{U_0}$ . Let $\Delta$ be the diagonal of $U\times U$. 
We have two projection maps $\pr_1,\pr_2$ from $\widehat{U\times U}_{/\Delta }$ to the first and second $ U$.
For $z\in U(k)$,  the restriction of $\pr_2$ induces  \begin{equation}\pr_1^{-1}(\{z\})\overset{\pr_2}\cong    \widehat{  U}_{/z} \label{633}.
\end{equation}
 
 Let $\CO\left(\widehat{U\times U}_{/\Delta }\right)$ be  the coordinate ring of $\widehat{U\times U}_{/\Delta }$.
Endow   $\CO\left(\widehat{U\times U}_{/\Delta }\right)$ a $B$-algebra structure via $\pr_1$.
  By Proposition \ref{geogl}, 
   we have a  (non-unique) $B$-algebra  isomorphism \begin{equation}\CO\left(\widehat{U\times U}_{/\Delta }\right)\cong  B[[T_1,...,T_{ g(g+1)/2}]].\label{BTT}\end{equation}

   Let
 $ A/\widehat{U\times U}_{/\Delta }$ be the pullback of   $A^{\univ}|_{U_0}$.  
Assume that 
$\Sym^2(T_p A^{\univ}[p^\infty]^\et) (U)$ is a free $\BZ_p$-modules of rank $ g(g+1)/2$. Let  $ \xi_1 ,..., \xi_{g(g+1)/2} $ be a basis of $\Sym^2(T_p A^{\univ}[p^\infty]^\et) (U)$ (whose existence follows from the definition of $\fI_\infty$ and the polarization). 
 The realization  of  the global  Serre-Tate  coordinate system of   $ A/\widehat{U\times U}_{/\Delta }$ at the basis $ \xi_1 ,..., \xi_{g(g+1)/2} $  is a construction of   an isomorphism \eqref{BTT} as follows.

For the  simplicity of notations, let us assume $g=1$.  The general case can be dealt in the same way. 
Let $\xi=\xi_1$ and $T=T_1$.
    Let  $$U'=\widehat{U\times U}_{/\Delta } /T^{p^n} \cong \Spec B[[T]]/T^{p^n} $$ and  let $A_n $ be the restriction of $A$ to $U'$. The  global  Serre-Tate  coordinate system  of  $A_n/U'  $ is a homomorphism of $\BZ_p$-sheaves over $U_{2,\et}$
 $$q_{A_n/U'}: \Sym^2(T_p A_n [p^\infty]^\et) \to \nu_{U'} .$$  
Note that
   $\xi$ gives a basis  
   $\xi_n$ of 
   $\Sym^2(T_p A_n [p^\infty]^\et)  (U')$.
Then we have 
$$q_{A_n/U'}(\xi_n)\in \nu(U')\cong\nu_{U}(U)\bigoplus
(1+TB[T]/T^{p^n}) $$
where the second isomorphism is  \eqref{phixB}.
Consider the morphism $$\phi_n: \nu(U')\to(1+TB[T]/T^{p^n})\incl B[T]/T^{p^n}$$
where the first map is the projection and second map is  the natural inclusion.   
Let $T_n^\ST\in  B[T]/T^{p^n}$ be  $ \phi_n(q_{A_n/U'}(\xi_n))-1 $.  
As $n$ varies, $T_n^\ST $'s  give an element $$T^\ST\in \CO\left(\widehat{U\times U}_{/\Delta }\right)\cong \vpl_n  B [T]/T^{p^n}.$$

 We compare the above construction with the realization of the classical Serre-Tate coordinate system.
Let $z\in U(k)$.
The restriction of $A$ to $\pr_1^{-1}(\{z\})$ is pullback  $A^\univ|_{\widehat{  U}_{/z }}$ of    $A^{\univ}|_{U_0}$ to  $\widehat{  U}_{/z }$ via \eqref{633}.
 (Thus we may regard $A$ as the  family $\{A^\univ|_{\widehat{  U}_{/z }}:z\in U(k)\}.$)
    The realization of the classical Serre-Tate coordinate system of  $\widehat{  U}_{/z }$
  at $\xi_z$ (the restriction of $\xi$ at $z$) gives an element $T_z^{c\ST}\in \widehat{  U}_{/z }$ and an  isomorphism 
  $ \widehat{  U}_{/z }\cong  \Spf k[[T_z^{c\ST}]]$
 (see Definition \ref{realization'}). Here  and below, the supscript $c$ indicates ``classical".

\begin{lem}\label{6.3.3}
The restriction of $T^\ST$ to $\pr_1^{-1}(\{z\}) \cong\widehat{  U}_{/z} $ is $T_z^{c\ST}$. In particular, $$\CO\left(\widehat{U\times U}_{/\Delta }\right)= B[[T^\ST]].$$
\end{lem}

\begin{proof} 

The restriction of  \eqref{BTT}  to $\pr_1^{-1}(\{z\}) \cong\widehat{  U}_{/z} $ via \eqref{633}  gives an isomorphism $\CO(\widehat{  U}_{/z} )\cong  k[[T]]. $ 
 Let $$q^c_n=q_{A^\univ|_{\widehat{  U}_{/z } }/T^{p^n}}^c: \Sym^2(T_p A^\univ_z[p^\infty])\to 1+T k[[T]] /T^{p^n}$$ be the    classical Serre-Tate coordinate system of $A^\univ|_{\widehat{  U}_{/z }} /T^{p^n}$ (see \eqref{60}).  Then   the image of $T_z^{c\ST}$ in $  k[[T]] /T^{p^n}$ is $q^c_n (\xi_z)-1.$
By Example \ref{GSTeg} (3) and  Lemma   \ref{GSTeg'}, $q^c_n(\xi_z)$ equals the restriction of  $\phi_n(q_{A_n/U'}(\xi_n))$   at $z$.   Thus the  first statement  follows. The second  statement  follows from the first one.
\end{proof}

          \section{Proof of Theorem \ref{TVcor1} }\label{pf{TVcor}}
In this section,  we at first prove a Tate-Voloch type result in   a family  in characteristic $p$.
 Combined with    the results in Section \ref{The ordinary perfectoid Siegel space and Serre-Tate theory}, we prove Theorem \ref{TVcor1}.
 We continue to use the notations  in Section \ref{The ordinary perfectoid Siegel space and Serre-Tate theory}.

    \subsection{Tate-Voloch type result  in   a family  in  characteristic $p$}\label{pf{TVcor1}}
    Recall that   $k=\bar \BF_p$   and  $K^\flat= k((t^{1/p^\infty})) $.   
    In the proof of  Lemma \ref{g=0}, we used the following simple fact: let   
    $S$   a $k$-algebra, $g\in S$ and  $x\in   (\Spec S)(k)$, then   $g (x)= 0$ or $|g (x)|_k=1$ where the   valuation $|\cdot|_k$ on $k$   takes value  0 on $0\in k$ and 1 on $k^\times$.   This fact can be naively regarded as an analog of the  Tate-Voloch conjecture over $k$. We want to consider this analog in a family. We need some notations.

       Let $l$ be a positive integer. 
 For $d=(d^{(1)},...,d^{(l)}) \in (p^{\BZ_{< 0}})^l$,  define $t^d:=(t^{d^{(1)}},...,t^{d^{(l)}})\in  (K^\fcc)^l$. 
 For $c=(c^{(1)},. . . ,c^{(l)}) \in ( {\BZ_{p}}^\times)^l$,  define $$(1+t^d)^c-1:=\left((1+t^{d^{(1)}})^{c^{(1)}}-1,. . .,(1+t^{d^{(l)}})^{c^{(l)}}-1\right)\in  (K^\fcc)^l.$$ 
 Fix     a sequence   $\{d_n\}_{n=1}^\infty$ of elements in  $(p^{\BZ_{< 0}})^l$ and   a sequence $\{c_n\}_{n=1}^\infty$  of elements in  $(\BZ_p^\times)^l$.
 Let 
 \begin{equation} \label{equation}y_n=  (1+t^{d_n})^{c_n}-1  \in (K^\fcc)^l\subset \Spec K^\fcc[[ T_1,...,T_l]]. 
  \end{equation}
Let $\BN=\{1,2,. . . \}$ the sequence of positive integers. For  $\delta\in (0,1)$ and   the  given sequence    $\{d_n\}_{n=1}^\infty$, let   $$\BN(\delta)=\{n\in \BN:
 d_n^{(i)}/d_n^{(i+1)}<\delta\}.$$    If $l=1$, we understand $\BN(\delta)$ as $\BN$.

\begin{prop} \label{7.2.1} Let $  A$ be a reduced $k$-algebra and  $V=\Spec A$.    
   Let $ \{z_n\}_{n=1}^\infty $ be a sequence of  (not necessarily distinct) points in $V(k)$. 
   Let $f\in A[[ T_1,...,T_l]]$ and 
let $f_{z_n}\in k[[ T_1,...,T_l]]$  be the  restriction      of $f$ at $z_n$.
Assume that  \begin{itemize} 
  \item[$(\star)$]  for every 
  infinite subset  $\BN'\subset \BN$, 
   the set
   $\{z_n: n\in \BN'\} $ is Zariski dense in $V$. 
\end{itemize}
If $f\neq 0$,
then   there exists    $  D_0\in \BR_{>0} $  and $\delta_0\in (0,1)$  such that for      every $D\geq D_0$ and $\delta\leq \delta_0$,   
   the following set is finite
  \begin{equation}  \{ n\in \BN(\delta) : \|f_{z_{n}}(y_{n})\| < \|T_l (y_n)\|^D\} \label{fDn}.\end{equation}  
 Here $T_l (y_n)$ is, by definition,  the $l$-th coordinate of $y_n$.

\end{prop}

\begin{proof}

We do induction on $l$. 

The case $l=1$  is proved as follows.
  Let $f=\sum\limits_{m\geq 0} a_mT^m$ where $ a_m \in A$. Regard $a_m$ as a function on $V$ so that $a_m(z_n)\in k$. 
Claim: there exists   some $m$ such that
$a_{m}({z_{n}})\neq 0$ for $n$ large enough. Let $m_0
$ be the smallest such $m$.
 Then $$\|f_{z_{n}}(y_{n})\|=\|t^{d_{n}m_0}\| $$   for $n$ large enough.
 Let $D_0=m_0$ and we are done. 
Now we prove the claim by contradiction. Assume that  for every  $m$,    $a_{m}(z_n)= 0$ for  infinitely many $n$.  By assumption $(\star)$ and  the reducedness of $A$, $a_m=0$.
    Thus $f=0$. This is a contradiction.

Now we do the induction. Let $l>1$. 
We prepare some notations.
Let $d'_n,y'_n$  be the first $l-1$ components of $d_n,y_n$  respectively. For $\delta\in (0,1)$, we have a subsequence  $\BN(\delta)'\subset \BN$ defined using  the sequence $\{d'_n\} _{i=1}^\infty$.
 Then  $\BN(\delta)'\supset \BN(\delta)$.

Assume that $f\neq 0$. Write $f=T_l^{m_1}(g_1+f_1)$  where $g_1\in A[[T_1,. . .,T_{l-1}]]\bsl \{0\}$ and $f_1\in  T_l A[[T_1,. . .,T_{l}]]$. 
  Below, to lighten notation, we abbreviate the subscript $z_n$. 
  Then for $n$ in the set  \eqref{fDn}, with $D$ and $\delta$ to be determined, we have    \begin{equation*}\|g_1(y_n)+f_1(y_n)\|= \|T_l (y_n)\| ^{-m_1}\|f(y_n)\|< \|T_l (y_n)\| ^{D-m_1} . \end{equation*}
 If $D\geq m_1+1$, then 
    $$\|g_1(y_n) \|\leq \|g_1(y_n)+f_1(y_n) \|+\|f_1(y_n)\|\leq   \|T_l (y_n)\|.$$
Since $\|T_{l} (y_n)\|< \|T_{l-1} (y'_n)\|^{1/\delta}$ and $\|g_1(y'_n)\|=\|g_1(y_n)\|$, we have 
 \begin{equation}
 \|g_1(y'_n)\|< \|T_{l-1} (y'_n)\|^{1/\delta}.
\label{woll}\end{equation}
By the  induction hypothesis,  there exists $D'>0$ and $\delta'_0\in(0,1)$ such that if  $ \delta\leq 1/D'$ and $\delta\leq \delta'_0$,  $\{n\in \BN(\delta)':\eqref{woll}\text{ holds}\}$ is finite.   Then   \eqref{fDn} is finite by choosing $\delta_0=\min\{1/D', \delta'_0\}$. \end{proof}
 
   \begin{rmk}(1)    $  D_0$ and  $\delta_0 $ are uniform for all choices of  $\{c_n\}_{n=1}^\infty$. We do not need this fact later. 
   
   (2)  The proposition  is inspired by   \cite[Lemma 2.10]{Ser}. In the proof of \cite[Lemma 2.10]{Ser}, there is a minor imprecision.  
   The following modification is suggested by Serban.
Define $T_\delta$ in  \cite[Lemma 2.10]{Ser} to be the first set in the intersection
but not the entire intersection, so that the statement (2) in loc. cit.  is about $T_\delta\cap S_{\phi}(q^{-1-c})$.
The 3rd displayed formula in the proof of \cite[Lemma 2.10]{Ser} should be removed.
     Then, on can still get the 5th displayed formula  in that proof with slightly more effort.  
     \end{rmk} 
     
  \subsubsection{Closure and limit} 
We show that  assumption $(\star)$ in Proposition \ref{7.2.1} holds in some situations.

                 \begin{lem}\label{1lem} 
                    Let $\{B_i\}_{i=0}^\infty$ be a  system of rings and $B=\vil_i B_i$. Let $  f_i:\Spec B\to \Spec B_i$ be the natural  morphism.
               Let $ \Lambda\subset \Spec B $ be a   subset and  $ \Lambda_i=  f_i(\Lambda)\subset \Spec B_i $.  We have the following relation between Zariski closures: \begin{equation}\Lambda ^\Zar =\bigcap_{i=0}^\infty   f_i^{-1}\left( \Lambda_i^\Zar\right) .\label{bydef}\end{equation}

\end{lem}
\begin{proof}
 The ideal $I\subset B$   defining $\Lambda ^\Zar$, with reduced induced structure as a closed subscheme, is generated by the union of the images $ I_i$ in $B$, where $I_i\subset B_i$ is the ideal of elements whose image in $B$ vanishes on $\Lambda ^\Zar$.  By the definition of $\Lambda_i$, $I_i$ is the ideal defining $\Lambda_i^\Zar$. Then \eqref{bydef} follows.
        \end{proof}

Let $f:U\to U_0$ be a  surjective morphism of  schemes. Let $\Lambda_0\subset U_0 $ be a subset  with Zariski closure $\Lambda_0^\Zar$  in $U_0$. 
 For $s\in \Lambda_0$, choose $z_s\in f^{-1}(s)$. Let $\Lambda=\{z_s:s\in \Lambda_0\}$ with Zariski closure $  \Lambda^\Zar$ in $U$.
 
 \begin{lem}   
 Assume that $f$ is closed.
 
 (1) 
 The image of $\Lambda^\Zar$ in $U_0$ is $\Lambda_0^\Zar$.
 
(2)    Assume that $\Lambda_0^\Zar$  is irreducible and $U$ is noetherian. 
There exists a choice of $\Lambda$ such that  $\Lambda^\Zar$  is irreducible. 

(3) In (2), further assume that $f$ is finite and the Zariski closure of every infinite subset of $\Lambda_0$ is $\Lambda_0^\Zar$. Then  the Zariski closure of every infinite subset of $\Lambda$ is $\Lambda^\Zar$. 
  \end{lem}
  \begin{proof} (1) is easy and the proof is omitted. 
  
  (2)    For every member of the finitely many irreducible  (so closed) components of $f^{-1}(\Lambda_0^\Zar)$, its image in 
  $\Lambda_0^\Zar$  is a closed  subscheme. By the irreducibility of $\Lambda_0^\Zar$,   some  irreducible component of $f^{-1}(\Lambda_0^\Zar)$   is surjective to $\Lambda_0^\Zar$.
We choose  all $z_s$'s in this component. 

(3) Note that a finite surjective morphism preserves dimension, and a proper closed subscheme of a noetherian irreducible scheme has a strictly smaller dimension. Then (3) follows from (1) and counting dimensions.   
 \end{proof}

The last two lemmas imply the following corollary.
  \begin{cor} \label{profirred}    Let
  $B,B_i$'s  be as in Lemma \ref{1lem}.  Let $ U=\Spec B$ (not necessary noetherian), $U_0=\Spec B_0$ and $f=  f_0$.  
  Assume that each $B_i$ is noetherian  and the transition morphisms $\Spec B_j\to \Spec B_i$ are finite surjective.
 Assume that   the Zariski closure of every infinite subset of $\Lambda_0$ is $\Lambda_0^\Zar$.
 There exists a choice of $\Lambda$ such that    the Zariski closure of every infinite subset of $\Lambda$ is $\Lambda^\Zar$.

  \end{cor}
  To fulfill the second assumption of the corollary, we use the following lemma.
  \begin{lem} \label{profirredlem}  Let $U_0$ be a noetherian scheme. For every infinite  subset   $Y\subset U_0 $, there is an   infinite  subset  $\Lambda_0\subset Y $ such that  the Zariski closure of every infinite subset of $\Lambda_0$ is $\Lambda_0^\Zar$. 

\end{lem} 
 \begin{proof}  By the noetherianness of   $U_0$, there exists a closed subscheme
 $V$ of  $U_0$  containing an   infinite  subset $\Lambda_0$  of $ Y $ such that every  proper
 closed subscheme of  $V$  only contains finitely many elements in  $Y$.  \end{proof} 
 
   \subsection{Proof of  Theorem \ref{TVcor1}} 
  
 Let $X$ be a product of Siegel moduli spaces over $\BZ_p$ with  certain  level structures away from $p$.
 By Lemma \ref{T'T},    Theorem \ref{TVcor1}  follows from 
the following theorem.

  \begin{thm} \label{TVcor1'}   Let $Z$ be a closed subvariety of 
of  $X_{\bar L} $. There exists a constant $c>0$ such that for every     ordinary CM point   $x \in X(L^\cyc)$, if $d(x,Z)\leq c$, then $x\in Z$.\end{thm}
Here the distance function $d(x,Z)$ is defined as in  \ref{The distance function} using the integral  model $X$
\begin{proof}
 We prove Theorem \ref{TVcor1'} when $X$ is a single Siegel moduli space. The general case is proved in the same way or by embedding a product of Siegel moduli spaces into a bigger one.
We continue to use the notations  in Section \ref{The ordinary perfectoid Siegel space and Serre-Tate theory}. In particular,   the 
fields $L, L^\cyc$, $K$ and $K^\flat$ below are as in the beginning of Section \ref{The ordinary perfectoid Siegel space and Serre-Tate theory};
   the formal  scheme $\fX(0)$, the   adic   locus $\CX(0)$, the perfectoid spaces $\CX(0)^\perf,\CX(0)^{\perf\flat}$ and Frobenius morphism $\Fr^\can$
   below are as in \ref{Ordinary perfectoid Siegel space}.
For an     ordinary CM point   $x \in X(L^\cyc)$,    we the same notation $x$ to denote
   its base change in $X(K)$.      
     Let    $x^\circ$ be the unique $K^\circ$-point in $ X$ whose generic fiber is $x$.

      Suppose that $Z$ is defined over a finite  Galois extension $F $ of  $L$. Let $\CI$ be the ideal sheaf of the schematic closure of $Z$ in $X_{ F^\circ}$.
   Let $\CU$ be an affine open subscheme of $X_{W}$,  of finite type over $W$.  (This is the only use of a calligraphic font not representing an adic space in this paper.)
 We only need to find a constant $c$ such that,  if  an 
     ordinary CM point   $x\in X(L^\cyc)$  satisfies $x^\circ\in \CU (K^\circ)$ and $d_{\CU_{K^\circ}}(x_{K},I)<c$, then $x\in Z$.
 Here the distance function is as in   \ref{The distance function}.

We at first have the following simplification on $F$.  Let $K'=FK$.  Suppose $\CI (\CU_{F^\circ})$ is generated by   $f_i$, $i=1,...,n$.
 For $\sigma\in G:=\Gal(K'/K) $, 
    $f_i^\sigma$  is in the coordinate ring of $\CU_{K'^\circ}$  and 
  $\|f_i(x_{K'} )\| =\|f_i^\sigma(x_{K'} )\| $.
Let $I$ be the ideal of the coordinate ring of $\CU_{K^\circ}$ generated by $\prod_{\sigma\in G} f_i^\sigma$, $i=1,...,n$. Then  $$d_{\CU_{K^\circ}}(x_{K},I)= d_{\CU_{K'^\circ}}(x_{K'},\CI_{K'^\circ}(\CU_{K'^\circ}))^{|G|}.$$   Thus  we may assume that $F\subset K$. Equivalently, $F\subset L^\cycl$.

  Now we reduce Theorem \ref{TVcor1'}  to  Theorem \ref{TVcor1aff} below, which is formulated with affine formal schemes.
  For an      ordinary CM point $x\in X(K)$, we  also use $x$ to denote the corresponding point   $  \CX(0)(K,K^\circ)$.  
Let $\fU$  be the  restriction of the $\varpi$-adic  formal completion   of $\CU$  to  $\fX(0)$. 
 By Lemma \ref{distep},  Theorem \ref{TVcor1'} is deduced from  Theorem \ref{TVcor1aff}. \end{proof}

 \begin{thm} \label{TVcor1aff}   Let $\fZ$ be an irreducible  closed formal subscheme of $\fU_{F^\circ}$. For  a  sequence   $\{x_n\}_{n=1}^{\infty}$  of ordinary CM points  such that $x_n$ is in the $\ep_n$-neighborhood of $\fZ$  and 
with $ \|\ep_n\|\to 0$,  we have  $x_n\in \CZ$ for  infinitely many $n$'s.
\end{thm} 
 
The proof of Theorem \ref{TVcor1aff}  consists of two bulks: one involvs
   perfectoid  spaces and one does not.  The perfectoid  one is    more technical and proves results to be used in the second one.  The non-perfectoid one   concludes Theorem \ref{TVcor1aff}.  We will present the on-perfectoid one     first,  in \ref{Global  Serre-Tate  coordinate} and  \ref{Proof of Theorem {TVcor1aff}}.
     
   A  canonical lifting  is an ordinary CM points of order  1 w.r.t.  a (equivalently every) basis, see Definition \ref{generator} (1). The following lemma will be proved in Theorem \ref{samel} using    perfectoid  spaces.

       \begin{lem}\label{samellem}
      Theorem \ref{TVcor1aff} holds  if we replace ``ordinary CM points" by ``canonical liftings".
\end{lem} 
   \subsubsection{Global  Serre-Tate  coordinate} \label{Global  Serre-Tate  coordinate} 
 Before we proceed to  the  proof of Theorem \ref{TVcor1aff}, let us recall  the realization of  the global  Serre-Tate  coordinate system  in \ref{GSTapp}.    

 Let   $U_0$ be  the special fiber of $\fU $. 
 Let $U=\Spec B$ be  
the profinite Galois cover of $U_0$ defined in \ref{GSTapp} (and coming from the infinite level Igusa scheme) 
such that  
$\Sym^2(T_p A^{\univ}[p^\infty]^\et) (U)$ is a free $\BZ_p$-modules of rank $ g(g+1)/2$.  Let  $\Delta$ be the diagonal of $U\times U$.
Then by Lemma \ref{6.3.3},  for  a basis   \begin{equation} \xi_1 ,..., \xi_{g(g+1)/2}  \label{base0}\end{equation}   of $\Sym^2(T_p A^{\univ}[p^\infty]^\et) (U)$, we have the  realization of the global  Serre-Tate coordinate system 
 $$\CO\left(\widehat{U\times U}_{/\Delta }\right)= B[[T_1^\ST,...,T_{g(g+1)/2}^\ST]],$$
 which has the following property. 
For every $z\in U(k)$, we have an isomorphism \begin{equation}\pr_1^{-1}(\{z\})\overset{\pr_2}\cong    \widehat{  U}_{/z} \label{633'}
\end{equation} as in \eqref{633}, and   the  corresponding  isomorphism 
\begin{equation} \widehat{  U}_{/z}\cong  \Spf k[[T_{1,z}^\ST,...,T_{g(g+1)/2,z}^\ST]],\label{7.2}\end{equation}
 Let  $T_{i,z}^\ST$ be the restriction  of   $T_i^\ST$ to $ \pr_1^{-1}(\{z\})\cong \widehat{  U}_{/z}$. 
Let  \begin{equation}\xi_{z,1},...,\xi_{z,g(g+1)/2}  \label{base}\end{equation}  be the restriction   of $ \xi_1 ,..., \xi_{g(g+1)/2} $. 
Then  \eqref{7.2}
 coincides with the realization of the classical Serre-Tate coordinate system of 
$ \widehat{  U}_{/z } $ at $\xi_{z,1},...,$ $\xi_{z,g(g+1)/2} $, see Definition \ref{realization'}.
  
   \subsubsection{Proof of Theorem \ref{TVcor1aff}} \label{Proof of Theorem {TVcor1aff}}

After passing to an infinite subsequence, we may assume that  $\{\red( x_n)\}_{n=1}^\infty$ is a sequence of the same point or pairwisely different points.
Let    $z_n\in U(k)$ be  over   $\red( x_n)\in U_0(k)$. 
By Corollary \ref{profirred}   and 
Lemma  \ref{profirredlem}, after passing to an infinite subsequence, we may assume the following.
\begin{asmp}\label{star}
For every 
  infinite subset  $\BN'\subset \BN$, 
   the Zariski closure  of the set
   $\{z_n: n\in \BN'\} $ in $U$ is the Zariski closure  of the set
   $\{z_n: n\in \BN\} $.   \end{asmp}

We regarded the  basis       \eqref{base} for $z=z_n$ as a basis of    $\Sym^2(T_p A_{x_n})$  naturally.
 Let $x_n$ be   of  order $p^{a_n}$   w.r.t.   \eqref{base}   (see Definition \ref{generator} (1))
where   $a_n=(a_n^{(1)},...,a_n^{(g(g+1)/2)})\in\BZ_{\geq 0}^{g(g+1)/2}$. 
 After passing to an infinite subsequence and permuting the basis  \eqref{base0}  of $\Sym^2(T_p A^{\univ}[p^\infty]^\et) (U)$, we  may assume that every $a_n$ is non-increasing  (see Definition \ref{generator} (2)). Let  $l\leq g(g+1)/2$ be a non-negative integer such that for every $n$, if  $i>l$, then $a_n^{(i)}=0$. 
  For example, if $l=g(g+1)/2$, the assumption   automatically holds;  if $l=0$, we are in the situation of Lemma \ref{samel1}.

 We will reduce Theorem \ref{TVcor1aff} to the case $l=0$ by using Lemma \ref{samel1} below. 
 We  need   the ``upper triangular change of  variables"  argument following \cite{Ser}.  
 By ``upper triangular change of  variables", we indeed mean changing the first $l$-element of the basis  \eqref{base0}  of $\Sym^2(T_p A^{\univ}[p^\infty]^\et) (U)$ via an upper triangular matrix as follows. 
For  $C\in \GL_l(\BZ_p)$, $(\xi_1 ,..., \xi_{l})C$ combined with $(\xi_{l+1},. . .,\xi_{(g(g+1)/2)})$ gives a new basis  of $\Sym^2(T_p A^{\univ}[p^\infty]^\et) (U)$. Thus by restriction as in \eqref{base}, we have a new basis of $\Sym^2(T_p A_{x_n})$ for every $n$.  
 Let $x_n$ be   of  order $p^{a_n(C)}$   w.r.t.  this new basis, where $a_n(C)\in\BZ_{\geq 0}^{g(g+1)/2}$. 
Then for  $C$   upper triangular, $a_n(C)$ is still non-increasing.
 
  \begin{lem}\label{samel1}
  Assume Assumption \ref{star}.
Assume that for every upper triangular matrix $C\in \GL_l(\BZ_p)$, the $l$-th component (so the $i$-th component for $i=1,. . .,l$ as well) of $a_n(C)$   goes to $\infty$ as $n\to \infty$.
Then    $ x_n\in \CZ$ for all $n\in \BN$.

 \end{lem}
 We postpone the proof of  Lemma \ref{samel1}.
 
We  finish the proof of Theorem \ref{TVcor1aff} by  induction on
  the  dimension   of $\fZ$. 
 If $\fZ$ is  empty, 
define its dimension to be $-1$.  When $\fZ$ is of dimension $-1$, the theorem is  trivial. 
The induction hypothesis is that the theorem holds for lower dimensions, and it will only be used in the proof of Lemma \ref{claim1} (2) below. 

By Lemma \ref{samel1} and passing to an infinite subsequence, we may assume that   for an upper triangular matrix $C\in \GL_l(\BZ_p)$, the $l$-th component of $a_n(C)$ is bounded. Replacing the basis \eqref{base0} by the new basis that is 
 $(\xi_1 ,..., \xi_{l})C$ combined with $(\xi_{l+1},. . .,\xi_{(g(g+1)/2)})$, we may assume that  there is a non-negative integer $m$  such that for every $n$,   $a_n^{(l)}\leq p^m$.  The fact that   $a_n^{(i)}=0$  for  $i>l$ does not change.

\begin{lem} \label{claim1}Let $m$ be a non-negative integer. 
Then the following hold.
 
(1) The adic generic fiber of  $(\Fr^\can)^m (x_n^\circ)$ is   in the $\ep_n$-neighborhood of the scheme theoretic image $(\Fr^\can)^m(\fZ)$ (see  \cite[2.3]{Kap}).

(2)  Assume that    $(\Fr^\can)^m (x_n^\circ)  \in (\Fr^\can)^m(\fZ(K^\circ))$
for infinitely many $n$'s, then    $x_n^\circ \in\fZ(K^\circ) $ for infinitely many $n$'s.
\end{lem}

\begin{proof}
To lighten the notations, assume that $m=1$. 

Consider  the closed  formal subscheme $(\Fr^\can)^{-1}(\Fr^\can(\fZ))$  of $\fU$ which contains $\fZ$. 
Then $x_n^\circ$ is contained in the $\ep_n$-neighborhood of  $(\Fr^\can)^{-1}(\Fr^\can(\fZ))$ 
by Lemma \ref{union} (1).
Then (1) follows from   the analog of Lemma \ref{pullback1} for formal schemes (which directly follows from Definition  \ref{defnnbhd}).

For (2), we prove it by contradiction. Let  $\{n_i\}\subset \BN$ be an infinite subsequence such that $\Fr^\can(x_{n_i}^\circ)  \in \Fr^\can(\fZ(K^\circ))$
and 
  $x_{n_i}^\circ\not\in \fZ(K^\circ)$ for $n_i$ large enough.  
In particular, $(\Fr^\can)^{-1}(\Fr^\can(\fZ))\neq\fZ $.   Thus by \cite[Proposition 2.10]{Kap}, it is not hard to show that
$$(\Fr^\can)^{-1}(\Fr^\can(\fZ))=\fZ\bigcup  \fZ'$$  such that   $\fZ' $ does not contain $\fZ$   and $x_{n_i}^\circ \in  \fZ'(K^\circ) $.
By Lemma \ref{intersect}, every $x_{n_i} $ is contained in the $\ep_{n_i}$-neighborhood of $\fZ\bigcap  \fZ'$.
Let  $\fZ_1$ be the union of irreducible components  of $\fZ\bigcap  \fZ'$  which  dominate 
$\Spf F^\circ$.
By Lemma \ref{union} (2), there exists $\delta\in K^{\circ}-\{0\}$ such that every $x_{n_i} $ is contained
 in the $\ep_{n_i}/\delta$-neighborhood of $\fZ_1$.  
Since every irreducible component of $\fZ_1$ has  dimension   less than the dimension of $\fZ$,
by the  induction hypothesis, we have  $x_{n_i} \in \fZ_1(K^\circ)\subset \fZ(K^\circ)$. This is a contradiction. 
 \end{proof}   
           
          By  \eqref{Qp} and Lemma \ref{claim1},   after passing to an infinite subsequence, we may assume that for every $n$, if $i\geq l$, then
           $a_n^{(i)}=0$, i.e. we may replace $l$ by $l-1$.           Continue this process, we may assume that $l=0$, i.e. $a_n^{(i)}=0$ for every  $n$ and $i$.  
          Now  Theorem \ref{TVcor1aff} follows from Lemma \ref{samel1}.

  \subsubsection{Canonical liftings and perfectoid  strategy}\label{Canonical liftings and perfectoid  strategy}   
  Now our remaining tasks are: proof of  Lemma \ref{samellem}  and  proof of  Lemma \ref{samel1}.
  For Lemma \ref{samellem}, we prove an ``almost effective" version of Theorem \ref{TVcor1aff} for canonical liftings.  In the proof, we use the ordinary perfectoid Siegel space and Scholze's approximation lemma, following a strategy of Xie \cite{Xie}. Our later proof of  Lemma \ref{samel1} involves a more complicated version of  this proof (which in particular uses the global Serre-Tate coordinate).
  
       Let  $\CX $  be the restriction of $\CX(0)^\perf$ to  the adic generic fiber of   $\fU_{K^\circ}$.  
Then $\CX=\Spa(R,R^+)$ where $(R,R^+)$ is a perfectod affinoid  $(K,K^\circ)$-algebra (there is no need to specify $R$ though it is easy to do so).
  The restriction of $\CX(0)^{\perf,\flat}$ to the adic generic fiber  
of $U_0\otimes K^\fcc$ is $\CX^\flat=\Spa(R^\flat,R^\fpl)$, the tilt of $\CX$. More concretely, it is given as follows: 
   let $S_m$ be the coordinate ring of 
 $(\Fr^m)^{-1}(U_0)$ with the natural inclusion $S_{m-1}\incl S_m$, and  $S=\bigcup_m S_m$, then $ R^{\flat+}$ is the $\varpi^\flat$-adic completion of $S \otimes K^{\flat\circ}$. Let  $\CX_m $ be the adic generic fiber  of $\Spec S_m\otimes K^\fcc$,   and $$\pi_m:\CX^\flat\to \CX_m$$   the natural projection.  Recall $\pi$ and $\pi'$ as defined in \eqref{natpi}. 
 Then $\pi_0=\pi'|_{\CX^\flat}$ (which has image in $\CX_0$). We abbreviate $\pi|_{\CX}$ as $\pi$  (which has image in the adic generic fiber of   $\fU_{K^\circ}$.)
 Let $\rho$ be the restriction of $\rho_{\CX(0)^\perf}$
(see  
\eqref{natrho}) to $\CX$.

For  $f\in \CO(\fU)$   in the defining ideal of $\fZ$,   regard $f$ as an element of $ R^+$ by the inclusion $\CO(\fU) \subset R^+$.
  For $c\in \BZ _{> 0}$, choose $g$ as in Lemma \ref{Corollary 6.7. (1)} (w.r.t.  $f$)  and choose a finite sum  $$ g_c=\sum\limits_{\substack{i\in \BZ[\frac{1}{p}]_{\geq 0},\\i< \frac{1}{p}+c}}g_{c,i}   \cdot (\varpi^\flat)^i $$ as in Lemma \ref{improveCorollary 6.7. (1)} where $g_{c,i}\in S$ for all $i$.
        There  exists a positive integer  $m(c)$ such that $g_{c,i}\in S_{m(c)}$ for all $i$ by the finiteness of the sum.       Let  $G_c:=g_c^{p^{m(c)}} $.
Then   we have the finite sum \begin{equation}G_c=\sum\limits_{\substack{i\in \BZ[\frac{1}{p}]_{\geq 0},\\i<  \frac{1}{p}+c}}G_{c,i}   \cdot (\varpi^\flat)^{p^{m(c)}i},\label{icmicm}
\end{equation}where $G_{c,i}=g_{c,i}^{p^{m(c)}}$.
By the construction of $S_n$'s,
we have $G_{c,i}\in S_0.$
 Let $I_c$ be the ideal of $S_0$ generated by $\{G_{c,i}:i\in  \BZ[\frac{1}{p}]_{\geq 0},i<  \frac{1}{p}+c\} $. 
By the noetherianness of $S_0$, there exists a positive integer $M$ such that \begin{equation}\sum_{c=1}^\infty I_c=\sum_{c=1}^MI_c.\label{icm'}\end{equation} 

  For  $y \in \fX(0)$ and $\tilde y \in \pi^{-1}(y )\subset \CX  $,   $|f(\tilde y)|=\|f( y)\|$.
  If $\|f( y)\|\leq \|\varpi\|^{\frac{1}{p}+M}$, by \eqref{2.1} and \eqref{modi}, we have $|g_c(  \pi_{m(c)} (\rho(\tilde y)))|\leq \|\varpi\|^{\frac{1}{p}+c}$ for $c=1,...,M$. 
So  for $c=1,...,M$, we have \begin{equation}|G_c(  \pi_{0} (\rho(\tilde y)))|=|G_c(  \pi_{m(c)} (\rho(\tilde y)))|\leq \|\varpi\|^{(\frac{1}{p}+c)p^{m(c)}}\label{icm''}.
\end{equation}

  \begin{thm}\label{samel}
Assume that  $\{f_1,...,f_r\}\subset \CO(\fU)$ generates the ideal defining $\fZ$. For each $f_j$, let $M_j$ be the $M$ as in \eqref{icm'} with $f$ replaced by $f_j$. Let $\BM=\max \{M_j:j=1,...,t\}$.
Let 
$ y$ be a  canonical lifting in the $\varpi ^{\frac{1}{p}+\BM}$-neighborhood of $\fZ$. Then $y\in \CZ$.
   
\end{thm} 
 \begin{proof} Apply Lemma \ref{6.2.3} (2) to $y$ with $a=1$.    Choose $\tilde y\in \pi^{-1}(y)   $ to be  as in Lemma \ref{6.2.3} (2).  
Then $\pi_{0} (\rho(\tilde y))=\red(y)\in U_0(k)$, where we understand $U_0(k)$ as a subset of $\fX(K^\flat,K^{\flat\circ})$ naturally. 
 Let $f=f_j$ and $M=M_j$ for some $j$. Then
  $ |f(\tilde y)|\leq \|\varpi\|^{\frac{1}{p}+M}$ and thus we have \eqref{icm''}.
    Similar to Lemma \ref{g=0},  by  \eqref{icm''} and  \eqref{icmicm},  we have $G_{c,i}( \pi_{0} (\rho(\tilde y)))=0$ for every $c=1,...,M$ and  corresponding   $i$'s.  
    By \eqref{icm'}, $I_c(  \pi_{0} (\rho(\tilde y)))=\{0\}$ for every  $c\in \BZ _{>0}$.  
So $$G_c(  \pi_{m(c)} (\rho(\tilde y)))=G_c(  \pi_{0} (\rho(\tilde y)))=0$$ for every  $c\in \BZ _{>0}$.
 Thus  $g_c(  \pi_{m(c)} (\rho(\tilde y)))=0$.   
By  \eqref{2.1} and \eqref{modi},     $ |f(\tilde y)|\leq \|\varpi\|^{\frac{1}{p}+c}$ for every  $c\in \BZ _{>0}$. Thus $ |f(\tilde y)|=0$. 
     \end{proof}
 
\begin{rmk}The effectivity of $\BM$ is essentially determined by the effectivity of the determination of the approximating function $g$ in Lemma \ref{Corollary 6.7. (1)}. However, Scholze's proof of lemma \ref{Corollary 6.7. (1)} uses ``almost ring theory" and is not effective. It is meaningful to ask if Lemma \ref{Corollary 6.7. (1)} can be made effective.
\end{rmk}
                    
  \subsubsection{Toward the proof of  Lemma \ref{samel1}} \label{Toward the proof}

  This paragraph closely mimics the proof of Theorem \ref{samel}.
Let notations be as above  Theorem \ref{samel} and let $y=x_n$.
For every $c=1,...,M$ and a corresponding   $i$, 
we want  to show that  $G_{c,i}( \pi_{0} (\rho(\tilde x_{n})))=0$.
Then by \eqref{icm'}, $I_c(  \pi_{0} (\rho(\tilde x_{n})))=\{0\}$ for every  $c\in \BZ _{>0}$.  
So $G_c(  \pi_{m(c)} (\rho(\tilde x_{n})))=G_c(  \pi_{0} (\rho(\tilde x_{n})))=0$ for every  $c\in \BZ _{>0}$.
 Thus  $g_c(  \pi_{m(c)} (\rho(\tilde x_{n})))=0$.   
By  \eqref{2.1} and \eqref{modi},     $ |f(\tilde x_{n})|\leq \|\varpi\|^{\frac{1}{p}+c}$ for every  $c\in \BZ _{>0}$. Thus $ |f(\tilde x_{n})|=0$. 
Let $f$ run over a finite set of generators of  the defining ideal of $\fZ$ and choose infinite subsequences successively, we have $ x_n\in \CZ$ for   infinitely many $n$'s. 

\subsubsection{Spaces}

 For $x\in U_0(k)$ (resp. $   U(k)$), let $\cD_x$ 
 be the adic generic fiber of the    formal completion of $   U_0 \otimes K^\fcc$ (resp. $   U \otimes K^\fcc$) at $x$. (This coincides with the definition in  \ref
 {Tilting canonical liftings}.) Equivalently, $\cD_x$ is the adic generic fiber of the    formal completion of $   \widehat {U_0}_{/x} \otimes K^\fcc$ (resp. $   \widehat {U}_{/x} \otimes K^\fcc$).
The following two diagrams   summarize  the adic spaces/k-schemes and morphisms between them that  we use:
   \begin{equation}  
    \xymatrix{
   	\CX   \ar[r]^{\rho } \ar[d]^{\pi}  & \CX^\flat  \ar[d]^{\pi_0} & & & \\
\CX(0)    &     \CX_0& \ar[l]_{(1)  \ \ \ \ \ } \coprod_{x\in U_0(k)}\cD_x& \ar[l] _{\ \ (2)}\coprod_{z\in U(k)}\cD_z &    \\
   &   U_0& \ar[l]_{(1')  \ \ \ \ \ } \coprod_{x\in U_0(k)}\widehat{U_0}_{/x}& \ar[l] _{\ \ (2')}\coprod_{z\in U(k)}\widehat{U}_{/z} &\ar[l] _{\eqref{633'}\ \ \ }\coprod_{z\in U(k)}\pr_1^{-1}(\{z\})
   \ar[r]^{\ \ \  \ \ \  (3)} &   \widehat{U\times U}_{/\Delta } \\ 
   }  \label{faca1}
\end{equation}
 Here the morphisms (1) (1') (3) are the natural inclusions. And the morphism (2), when restricted to $\cD_z, z\in U(k)$, is the natural isomorphism  $\cD_{z}\cong\cD_x$  where
 $x\in U_0(k)$ is the image of $z$. We have  the parallel statement for (2').
 
\subsubsection{Functions}


 Let $H_{c,i}$ be the image of $G_{c,i}$ in $B$ under  the morphism $S_0=\CO(U_0)\to B =\CO(U)$,  and $H_{c,i,z_{n}}\in  \CO(\widehat{  U}_{/z_n}) $   the image of  $H_{c,i}$ under the morphism $B=\CO(U)\to  \CO(\widehat{  U}_{/z_n}) $.  
 
   For $\tilde x_n \in \pi^{-1}(x_{n})\subset \CX$, by  Lemma \ref{6.2.3} (1), $\pi_{0} (\rho(\tilde x_n))\in \cD_{\red( x_n)}$. Let  $y_n$ be the preimage of  $\pi_{0} (\rho(\tilde x_n))$ in  $\cD_{z_n}$ via the natural isomorphism  $\cD_{z_n}\cong\cD_{\red(x_n)}$.
Then as  elements in $K^{\fcc}$, we have 
$$H_{c,i,z_{n}}(y_n)=H_{c,i}(y_n)= G_{c,i}(\pi_{0} (\rho(\tilde x_{n}))) .$$ 
\begin{lem}\label{G<1} There is a constant $h_{c,i}<1$ such that $\|H_{c,i,z_{n_m}}(y_{n_m})\|<h_{c,i}$.
\end{lem}

\begin{proof}If the lemma is not true, let $i_0$ be the smallest  $i$ appearing in  the finite sum \eqref{icmicm} such that $\|H_{c,i,z_{n_m}}(y_{n_m})\|\to 1$
for  a subsequence $\{n_m\}_{m=1}^\infty\subset \BN$.  Then  \eqref{icmicm} implies that
 $\|G_{c}(  \pi_{0} (\rho(\tilde x_{n_m})))\|   \to \|\varpi\|^{i_0 p^{m(c)}},$ which contradicts \eqref{icm''}.  
 \end{proof}

Let $\phi$ be the  composition  of
\begin{equation*}\phi:   B=\CO(U)\to B\otimes B\to  \CO\left(\widehat{U\times U}_{/\Delta }\right)=   B[[T_1^\ST,...,T_{g(g+1)/2}^\ST]] \end{equation*}
where the first  morphism is $b\mapsto 1\otimes b$. 
I.e. $\phi$ gives the projection   $\pr_2:\widehat{U\times U}_{/\Delta }$ to the second $U$.
Tracking the second diagram  of \eqref{faca1}, we have the following lemma.
\begin{lem}\label{restriction }
    The restriction of $ \phi(H_{c,i})$  to $ \pr_1^{-1}(\{z_n\})\overset{\pr_2}\cong \widehat{  U_1}_{/z_n}$ in \eqref{633'} is $H_{c,i,z_{n}}\in  \CO(\widehat{  U_1}_{/z_n}) $.
\end{lem}
\subsubsection{Proof of Lemma \ref{samel1}}
We need some notations.
For an open subset $O\subset   \BZ_p^{l-1}$, 
 let $\BN(O)\subset \BN$ be the  subsequence   such that the first $l-1$ components of a  ratio of $x_n$  w.r.t. this basis  (see Definition \ref{generator}) is in $O$.    If $l=1$, we understand $\BN(O)$ as the whole  $ \BN$ (and we will not need the case $l=0$).
For $r \in  \BZ_p^{l-1}$ and $\delta\in (0,1)$, let $\BN(r,\delta)=\BN(O(r,\delta))$ where $O(r,\delta)$ is the $p$-adic closed disc centered at $r$ of radius   $\delta$.

Now we start to prove Lemma \ref{samel1}.
By the discussion in \ref{Toward the proof}, we only need to
  prove that for every $n\in \BN$, $G_{c,i}(\pi_{0} (\rho(\tilde x_{n}))) =0$.    Let $\Spec A\subset \Spec B$ be the Zariski closure of the set
   $\{z_n: n\in \BN\} $. 
Let $f$ be   the image of $\phi(H_{c,i})$ under $B[[T_1^\ST,...,T_{g(g+1)/2}^\ST]]\to A[[T_1^\ST,...,T_{g(g+1)/2}^\ST]]$.  
By  Lemma \ref{restriction }, we have 
\begin{equation}\label{GH0}G_{c,i}(\pi_{0} (\rho(\tilde x_{n}))) =H_{c,i,z_{n}}(y_{n})=f(y_n).\end{equation}
We prove the stronger result  $f=0$ by contradiction.

 Assume that $f\neq 0$. We want to apply   Proposition \ref{7.2.1} to $f$ and $y_n$'s. 
 We check the conditions in Proposition \ref{7.2.1}. First, by  the compatibility between the 
Global  and classical Serre-Tate  coordinates as in the end of  \ref
 {Global  Serre-Tate  coordinate}, we   use  Lemma \ref{6.2.3} (2) to conclude that $y_n$'s are as in \eqref{equation} above Proposition \ref{7.2.1}. 
Second, the assumption $(\star)$ in Proposition \ref{7.2.1} holds by    Assumption \ref{star}. 
By the assumption that $a_n$   goes to $\infty$ as $n\to \infty$ in Lemma \ref{samel1},    Lemma \ref{G<1} and  the second ``=" of \eqref{GH0}, for $n$ large enough, $n$ satisfies the inequality in \eqref{fDn} of  Proposition \ref{7.2.1} (for every $D$). 
  Then by Proposition \ref{7.2.1},
   there  exists $\delta_0\in (0,1)$ such that   $\BN(0,\delta_0) $  is finite.  
For a general $r \in  \BZ_p^{l-1}$,  by    \cite[Lemma 2.7]{Ser}, after an ``upper triangular change of  variables" (as defined above Lemma \ref{samel1}), we may use the same proof for $r=0$ to conclude that   there  exists $\delta_r\in (0,1)$ such that   $\BN(r,\delta_r) $  is finite.   
By its compactness,   $ \BZ_p^{l-1}$ is the union of  
 $p$-adic closed discs centered at $r$ of radius   $\delta_r$ for finitely many $r$'s. 
Then  the infinite set $\BN$ is the union of the finite sets  $\BN(r,\delta_r) $'s for these finitely many $r$'s. This is a contradiction.

\end{document}